\documentclass[11pt]{amsart}
\usepackage[OT2,T1]{fontenc}
\usepackage[numbers,sort&compress]{natbib}
\usepackage{amsmath,amsfonts,amsthm,amssymb,mathrsfs,
amsxtra,amscd,color,latexsym,color,paralist,bbm,verbatim}
\usepackage[hmargin=2.5cm,vmargin=3.5cm]{geometry}
\usepackage[normalem]{ulem}
\usepackage{soul}
\usepackage[utf8]{inputenc} 
\usepackage[T1]{fontenc}
\DeclareSymbolFont{cyrletters}{OT2}{wncyr}{m}{n}
\DeclareMathSymbol{\Sha}{\mathalpha}{cyrletters}{"58}

 \newtheorem{thm}{Theorem}[section]
 \newtheorem*{thm*}{Theorem}

 \newtheorem{lem}[thm]{Lemma}
 
 \newtheorem{conj}[thm]{Conjecture}

 \newtheorem{rmk}[thm]{Remark}

 \theoremstyle{definition}
 \theoremstyle{remark}
 \numberwithin{equation}{section}

\newcommand{\sm}{\left(\begin{smallmatrix}}
\newcommand{\esm}{\end{smallmatrix}\right)}
\newcommand{\mat}{\left(\begin{matrix}}
\newcommand{\emat}{\end{matrix}\right)}

\def\CC{\mathbb{C}}

\def\Q{\mathbb{Q}}

\def\ZZ{\mathbb{Z}}

\def\det{\mathrm{det}}

\def\Frob{\mathrm{Frob}}

\def\Gal{\mathrm{Gal}}
\def\Sh{\mathrm{Sh}}

\def\tr{\mathrm{tr}}

\def\GL{\mathrm{GL}}

\def\SL{\mathrm{SL}}

\begin{document}

\title[Newman's conjecture for the partition function]{Newman's conjecture for the partition function modulo integers with at least two distinct prime divisors}

\author{Dohoon Choi}
\author{Youngmin Lee}

\address{Department of Mathematics, Korea University, 145 Anam-ro, Seongbuk-gu,  Seoul 02841, Republic of Korea}
\email{dohoonchoi@korea.ac.kr}	

\address{School of Mathematics, Korea Institute for Advanced Study, 85 Hoegiro, Dongdaemun-gu,
Seoul 02455, Republic of Korea}
\email{youngminlee@kias.re.kr}



\begin{abstract}
Let $M$ be a positive integer and $p(n)$ be the number of partitions of a positive integer $n$.
Newman's Conjecture asserts that for each integer $r$, there are infinitely many positive integers $n$ such that
\[ p(n)\equiv r \pmod{M}. \]
For a positive integer $d$, let $B_{d}$ be the set of positive integers $M$ such that the number of prime divisors of $M$ is $d$.
In this paper, we prove that for each positive integer $d$, the density of the set of positive integers $M$ for which Newman's Conjecture holds in $B_{d}$ is $1$. Furthermore, we study an analogue of Newman's Conjecture for weakly holomorphic modular forms on $\Gamma_0(N)$ with nebentypus, and this applies to $t$-core partitions and generalized Frobenius partitions with $h$-colors.
\end{abstract}	
\maketitle
\section{Introduction}\label{s : intro}
A sequence of positive integers $\lambda=(\lambda_1,\dots, \lambda_i)$ is called a partition of a positive integer $n$ if $\lambda_1\geq\dots \geq \lambda_i$ and $\sum_{j=1}^{i} \lambda_j=n$.
Let $p(n)$ denote the number of partitions of $n$.
For example, $p(4)=5$ since
\[ 4=3+1=2+2=2+1+1=1+1+1+1.  \]
For convenience, let $p(0):=1$. In \cite{RA1, RA2, RA}, Ramanujan found the following three congruences for $p(n)$: for all non-negative integers $n$,
\[ p(5n+4)\equiv 0 \pmod{5}, \]
\[ p(7n+5)\equiv 0 \pmod{7}, \]
\[ p(11n+6)\equiv 0 \pmod{11}.\] 
These congruences have inspired many works on the arithmetic properties of $p(n)$: for example \cite{AH,AB1,AB2,AN2,AG,AT,AO2,AS,BO,FKO,KO1,KL,KO,LO,NE,ON1}. Especially,  Ono \cite{ON1} proved that if $\ell$ is a prime larger than or equal to $5$, then there are infinitely many non-nested arithmetic progressions $\{an+b\}$ such that for all non-negative integers $n$,
$$p(an+b)\equiv 0 \pmod{\ell}.$$
This reestablishes a conjecture of Erd\H{o}s and Ivi\'{c} \cite{EI}, and gives an answer to a conjecture of Erd\H{o}s.
Erd\H{o}s and Ivi\'{c} \cite{EI} conjectured that there are infinitely many primes $\ell$ dividing some value of the partition function, and Schinzel established the Erd\H{o}s and Ivi\'{c} conjecture by the Hardy-Ramanujan-Rademacher asymptotic formula for $p(n)$ (see \cite{EI} for the proof). Erd\H{o}s suggested a strong conjecture in his private communications with Gordon and Ivi\'{c} that
if $\ell$ is a prime, then there exists a non-negative integer $n_{\ell}$ such that
$$
p(n_{\ell})\equiv 0 \pmod{\ell}.
$$

In this vein, Newman \cite{NE} asserted the following conjecture for general residue classes modulo a positive integer.

\begin{conj}[Newman's Conjecture]\label{conj : conj 1}
Let $M$ be a positive integer.
Then, for each integer $r$, there are infinitely many positive integers $n$ such that
\[ p(n)\equiv r\pmod{M}. \]
\end{conj}
Kolberg \cite{KO} proved that Newman's Conjecture holds for $M=2$, and Newman \cite{NE} showed that it holds for $M=2,5$ and $13$. For other integers $M$, it was proven that Newman's Conjecture holds for 
$M=7$ by Atkin \cite{AT} and Kl\o{}ve \cite{KL}, and for $M=17,19,29$ and $31$ by Kl\o{}ve \cite{KL}.
For each prime $M\geq 5$, Ono \cite{ON1} proved that if for each integer $r$, there is a positive integer $n_r$ such that
\begin{equation*}
    p\left(\frac{Mn_r+1}{24}\right)\equiv r \pmod{M},
\end{equation*}
then
\begin{equation}\label{e : ono_bound}
     \#\{ n\leq X : p(n)\equiv r \pmod{M} \} \gg_{M} \begin{cases}
X \quad &\text{if }r=0,\\
\frac{\sqrt{X}}{\log X} \quad &\text{if }r\neq 0.
\end{cases}
\end{equation}
From this, he verified that Newman's Conjecture holds for every prime $M$ such that $M<1000$ with the exception of $M=3$.
The bounds \eqref{e : ono_bound} were generalized by Ahlgren \cite{AH} to positive integers $M$ relatively prime to $6$.
From these results, a proof of Newman's Conjecture for $M$ relatively prime to $6$ can be obtained by showing that for each integer $r$, there is a positive integer $n$ such that
\begin{equation}\label{e : eq of newman}
    p(n)\equiv r\pmod{M}.
\end{equation}

In \cite{BO}, Bruinier and Ono proved that if $M$ is a prime with $M>11$, and there is a positive integer $n$ such that
\begin{equation}\label{e : cond of b-o}
     p\left(\frac{Mn+1}{24}\right)\not\equiv 0\pmod{M},
\end{equation}
then for each integer $r$, there is a positive integer $n$ satisfying \eqref{e : eq of newman}.
The existence of a positive integer $n$ satisfying (\ref{e : cond of b-o}) was conjectured by Ramanujan \cite{RA2}. This conjecture was proved by Ahlgren and Boylan \cite{AB1}, and for primes $M \geq5$, they proved that Newman's Conjecture holds for $M$.
Furthermore, Ahlgren and Boylan \cite{AB2} obtained bounds for the weights of modular forms of half-integral weight having the Fourier expansion of special forms which were studied by Bruinier and Ono \cite{BO}. As a result, they proved that Newman's Conjecture holds for $M$ that is a power of a prime larger than or equal to $5$.

For $M$ having at least two prime divisors, to the best of the authors' knowledge, it is not known whether there are infinitely many $M$ such that Newman's Conjecture holds for $M$. We state our main theorem on the density of positive integers $M$ for which Newman's Conjecture holds.
\begin{thm}\label{thm : thm 2}
Assume that $\mathcal{N}$ is the set of positive integers $M$ such that for each integer $r$,
\[ \#\{ n\leq X : p(n)\equiv r \pmod{M} \} \gg_{M} \begin{cases}
X \quad &\text{if }r=0,\\
\frac{\sqrt{X}}{\log X} \quad &\text{if }r\neq 0.
\end{cases}\]
Assume that for each positive integer $d$, $B_{d}$ is the set of positive integers $M$ such that the number of prime divisors of $M$ is $d$. Then,
\[ \lim_{X\to \infty} \frac{\#\{M\leq X :  M\in B_{d} \cap \mathcal{N} \}}{\#\{ M\leq X : M\in B_{d} \}}=1. \]
\end{thm}

For any complex number $z$ with $\Im(z)>0$, let $q:=e^{2\pi i z}$ and $\eta(z)$ be Dedekind's eta function defined by
\[ \eta(z):=q^{\frac{1}{24}}\prod_{n=1}^{\infty}(1-q^n).\]
The generating function of $p(n)$ is given by
\[ \sum_{n=0}^{\infty} p(n)q^n=\prod_{n=1}^{\infty}\frac{1}{1-q^n}. \]
Then, we have
\[ \frac{1}{\eta(24z)}=\sum_{n=0}^{\infty} p(n)q^{24n-1}.  \]
Note that $\frac{1}{\eta(24z)}$ is a weakly holomorphic modular form of weight $-\frac{1}{2}$ on $\Gamma_0(576)$ with nebentypus $\left(\frac{12}{\cdot}\right)$. Our proof of Theorem \ref{thm : thm 2} is given by considering the distribution of the Fourier coefficients of a weakly holomorphic modular form with half-integral weight modulo a positive integer.

For an integer $k$, let $f$ be a weakly holomorphic modular form of weight $\frac{k}{2}$ on $\Gamma_0(N)$ with nebentypus $\chi$.
For each $\rho:=\left(\begin{smallmatrix} a & b\\
c & d\end{smallmatrix}\right)\in \SL_2(\ZZ)$,
let
\[ (cz+d)^{-\frac{k}{2}}f\left(\frac{az+b}{cz+d}\right)
:=\sum_{n\gg -\infty} a_{f}\left(\rho:\frac{n}{w_{\rho}}\right)q^{\frac{n}{w_{\rho}}} \]
for some positive integer $w_{\rho}$.
Let $\Omega(f)$ be the subset of $\mathbb{Q}^{\times}$ consisting of negative $\frac{n}{w_{\rho}}$ such that $a_{f}\left(\rho:\frac{n}{w_{\rho}}\right)\neq 0$ for some $\rho\in \SL_2(\ZZ)$, and $\overline{\Omega(f)}$ be the image of $\Omega(f)$ in $\Q^{\times}/\Q^{\times^2}$.
For example, if $f=\frac{1}{\eta(24z)}$, then $$\overline{\Omega(f)}=\left\{-\Q^{\times^2}\right\}$$
(for details, see \eqref{eq 4} in Section 4).
Let $\alpha:=\frac{(k,2)}{2}$ and $\mathcal{N}(f)$ be the set of positive integers $M$ such that for each integer $r$,
\[    \#\{ n\leq X : a_{f}(n)\equiv r \pmod{M} \} \gg_{M} \begin{cases}
X \quad &\text{if }r=0,\\
\frac{X^{\alpha}}{\log X} \quad &\text{if }r\neq 0.
\end{cases} \]
Theorem \ref{thm : thm 2} is implied by the following theorem on the distribution of the Fourier coefficients of a weakly holomorphic modular form modulo a positive integer.
\begin{thm}\label{thm : thm of weak}
Let $k$ be an integer, and $N$ be a positive integer such that $4|N$ if $k$ is odd. 
Let $f$ be a weakly holomorphic modular form of weight $\frac{k}{2}$ on $\Gamma_0(N)$ with nebentypus $\chi$.
Let $d$ be a positive integer and $B_{d}$ be as in Theorem \ref{thm : thm 2}.
Assume that $f$ has integral coefficients and is not a linear combination of single-variable theta series. If we let
\[ s(f):=\# \overline{\Omega(f)},\]
then
\begin{equation}\label{eq 3}
 \liminf_{X\to \infty} \frac{\#\{M\leq X :  M\in B_{d} \cap \mathcal{N}(f) \}}{\#\{ M\leq X : M\in B_{d} \}}\geq \left(\frac{1}{2}\right)^{d(s(f)-1)}.
\end{equation}
\end{thm}

\begin{rmk}
Assume that $f=\frac{1}{\eta(24z)}$.
Then, we have 
\[s(f)=\#\overline{\Omega(f)} =\# \left\{-\Q^{\times^2}\right\} = 1.\] 
This implies the right side of \eqref{eq 3} is equal to $1$.
Thus, Theorem \ref{thm : thm 2} follows from Theorem \ref{thm : thm of weak}.
\end{rmk}

When $f$ is a modular form on $\Gamma_0(N)$ with nebentypus $\chi$, Theorem \ref{thm : thm of weak} is improved as follows.

\begin{thm}\label{thm : thm 1}
Let $k$ and $N$ be positive integers with $4\mid N$. 
Let $f(z):=\sum_{n=0}^{\infty} a(n)q^n$ be a modular form of weight $k+\frac{1}{2}$ on $\Gamma_0(N)$ with nebentypus $\chi$.
Assume that $a(n)$ is integral for all $n$ and that $f$ is not a linear combination of single-variable theta series. Let $n_0$ be a square-free integer such that $n_0\nmid N$ and that $a(n_0m^2)\neq 0$ for some positive integer $m$.
If $M$ is relatively prime to $Nn_0ma(n_0m^2)$,
then for each integer $r$,
\begin{equation}\label{e : eq of well-distributed}
    \#\{ n\leq X : a(n)\equiv r \pmod{M} \} \gg_{M} \begin{cases}
X \quad &\text{if }r=0,\\
\frac{\sqrt{X}}{\log X} \quad &\text{if }r\neq 0.
\end{cases}
\end{equation}
\end{thm}

\begin{rmk}
Let us note the known results for the distribution of the Fourier coefficients of modular forms of integral or half-integral weight modulo an integer $M$. Let $f(z):=\sum_{n=1}^{\infty} a(n)q^n$ be a cusp form of integral or half-integral weight on $\Gamma_0(N)$ with nebentypus $\chi$ such that its Fourier coefficients are integral.
\begin{enumerate}
\item Assume that the weight of $f$ is integral.
Let $c$ be a positive integer.
For this case,  Serre \cite{SE} proved that if $M$ is an odd integer, then for each integer $r$,
\[\#\{ n\leq X : a(n)\equiv r \pmod{M} \} \gg_{M,c} \begin{cases}
X \quad &\text{if }r=0,\\
\frac{X(\log \log X)^{c}}{\log X} \quad &\text{if }r\neq 0.
\end{cases}\]
\item  Assume that the weight of $f$ is half-integral, and that $M$ is a power of an odd prime. Let $k+\frac{1}{2}$ be the weight of $f$, and $M:=\ell^{j}$, where $\ell$ is an odd prime and $j$ is a positive integer. For this case, Bruinier and Ono \cite{BO} proved that if there are not finitely many square-free positive integers $n_i$ such that
\begin{equation}\label{e : eq of pseudo theta form}
    f(z)\equiv \sum_{i=1}^{t} \sum_{m=1}^{\infty} a(n_im^2)q^{n_im^2} \pmod{\ell},
\end{equation}
then for each integer $r$, the bound \eqref{e : eq of well-distributed} holds for $M$.
Ahlgren and Boylan \cite{AB2} obtained the bounds for the weight of $f$ such that $f$ has the form \eqref{e : eq of pseudo theta form}. This implies that if $k>1$ and $\ell>2k+1$, then for each integer $r$, the bound \eqref{e : eq of well-distributed} holds for $M$.
\end{enumerate}
\end{rmk}

\subsection{Further applications of Theorems \ref{thm : thm of weak} and \ref{thm : thm 1}}
As applications of Theorems \ref{thm : thm of weak} and \ref{thm : thm 1}, we study generalized Frobenius partitions with $h$-colors and $t$-core partitions.

\subsubsection{Generalized Frobenius partitions with $h$-colors}
Let us give the definition of the generalized Frobenius partition of an integer $n$, which was given by Andrews \cite{AN}. A two-rowed array of non-negative integers
\[ \begin{pmatrix}
a_1 & \cdots & a_r \\
b_1 & \cdots & b_r
\end{pmatrix}\]
is called a generalized Frobenius partition of $n$ if $a_{i}\geq a_{i+1}$, $b_{i}\geq b_{i+1}$ and $r+\sum_{i=1}^{r} (a_i+b_i)=n$.
For a positive integer $h$, let $c\phi_{h}(n)$ denote the number of generalized Frobenius partitions of $n$ with $h$-colors (for details, see \cite{AN}).
The functions $c\phi_{h}(n)$ have been studied by many authors; for example \cite{AAD, BS, CGH, CWL,GS, GWX,L,Love,P, PR,WZ}.
In \cite[Theorem 5.2]{AN}, the generating function of $c\phi_{h}(n)$ is given by
\[  C\Phi_h(q):=\sum_{n=0}^{\infty} c\phi_h(n)q^n=\frac{1}{\prod_{n=1}^{\infty}(1-q^n)^h}\sum_{m\in \mathbb{Z}^{h-1}}q^{Q(m)}, \]
where $Q(m_1,\dots, m_{h-1}):=\sum_{i=1}^{h-1} m_i^2+\sum_{1\leq i<j\leq h-1} m_im_j$.
By \cite[Theorem 4.9.3]{MI}, we see that $q^{-h}C\Phi_{h}(q^{24})$ is a weakly holomorphic modular form of half-integral weight (for details, see Section \ref{s : section of weak}).
Thus, we have the following theorem.

\begin{thm}\label{thm : thm 3}
Let $d$ and $h$ be positive integers with $h>1$. Let $B_{d}$ be as in Theorem \ref{thm : thm 2}, and let $\mathcal{N}_{h}$ be the set of positive integers $M$ such that for each integer $r$,
\[ \#\{ n\leq X : c\phi_{h}(n)\equiv r \pmod{M} \} \gg_{M} \begin{cases}
X \quad &\text{if }r=0,\\
\frac{\sqrt{X}}{\log X} \quad &\text{if }r\neq 0.
\end{cases}\]
Then,
\[ \liminf_{X\to \infty} \frac{\#\{M\leq X :  M\in B_{d} \cap \mathcal{N}_{h} \}}{\#\{ M\leq X : M\in B_{d} \}}\geq \left(\frac{1}{2}\right)^{d(576h^2-1)}.\]
In particular, if $h=2$ or $h=3$, then
\[ \lim_{X\to \infty} \frac{\#\{M\leq X :  M\in B_{d} \cap \mathcal{N}_{h} \}}{\#\{ M\leq X : M\in B_{d} \}}=1. \]
\end{thm}

\subsubsection{$t$-core partitions}\label{sss : t-core}
A partition $\lambda$ of $n$ is called a $t$-core partition of $n$ if the hook length of each cell in the Ferrers diagram of $\lambda$ is not divided by a fixed positive integer $t$.
Let $c_{t}(n)$ denote the number of $t$-core partitions of $n$ (for details, see \cite{CKNS}).
Note that the generating function of $c_t(n)$ is given by
\[ \sum_{n=0}^{\infty} c_t(n)q^{n}=\prod_{n=1}^{\infty}\frac{(1-q^{tn})^t}{1-q^n}.\]
Let
\begin{equation*}
     f_{t}(z):=\sum_{n=0}^{\infty} c_{t}(n)q^{24n+t^2-1} = \frac{\eta^{t}(24tz)}{\eta(24z)}.
\end{equation*}
Then, $f_t$ is a modular form of integral or half-integral weight. Thus, we have the following theorem.

\begin{thm}\label{cor : cor 1}
Let $t$ be a positive integer with $t>2$.
Then, the following statements are true.
\begin{enumerate}
    \item Assume that $t$ is even.
    If $M$ is relatively prime to $3t(t^2-1)(t^2+23)$, then for each integer $r$,
\[\#\{ n\leq X : c_{t}(n)\equiv r \pmod{M} \} \gg_{M} \begin{cases}
X \quad &\text{if }r=0,\\
\frac{\sqrt{X}}{\log X} \quad &\text{if }r\neq 0.
\end{cases}\]
    \item Assume that $t$ is odd. Let $c$ be a positive integer. If $M$ is relatively prime to $3t(t^2-1)$, then for each integer $r$,
 \[ \#\{ n\leq X : c_t(n)\equiv r \pmod{M} \} \gg_{M,c} \begin{cases}
    X \quad &\text{if }r=0,\\
    \frac{X(\log \log X)^{c}}{\log X} \quad &\text{if }r\neq 0.
    \end{cases}\]
\end{enumerate}
\end{thm}

\subsection{Organization of this paper}

The remainder of this paper is organized as follows.
In Section \ref{s : pre}, we review the basic properties of modular forms such as the Hecke operator and Shimura Correspondence.
In Section \ref{s : distribution of the Fourier coefficients}, we study the distribution of the Fourier coefficients of modular forms of half-integral weight modulo an integer, and we prove Theorems \ref{thm : thm 1} and \ref{cor : cor 1}.
In Section \ref{s : section of weak}, we study the distribution of the Fourier coefficients of weakly holomorphic modular forms modulo an integer, and we prove Theorems \ref{thm : thm 2}, \ref{thm : thm of weak}, and \ref{thm : thm 3}.

\section{Preliminaries}\label{s : pre}
In this section, we review some notions and properties about modular forms.
Let $(a,b)$ be the greatest common divisor of two integers $a$ and $b$.
Let $k$ and $N$ be integers, and $\chi$ be a Dirichlet character modulo $N$.
Let $w$ be a rational number with $2w\in \mathbb{Z}$.
Assume that $4\mid N$ if $w$ is not an integer. Let $M_{w}(\Gamma_0(N),\chi)$ (resp. $S_{w}(\Gamma_0(N),\chi)$) be the space of modular forms (resp. cusp forms) of weight $w$ on $\Gamma_0(N)$ with nebentypus $\chi$.
Here, a weakly holomorphic modular form is a holomorphic function on the complex upper half-plane that satisfies the automorphy conditions in the definition of a modular form but may have poles at the cusps.
For a weakly holomorphic modular form $f$ of integral or half-integral weight, let $a_{f}(n)$ be the $n$-th Fourier coefficient of $f$.
Let $\left(\frac{\cdot}{\cdot}\right)$ denote the Kronecker symbol.
Let $\mathbb{H}$ be the upper half-plane of $\mathbb{C}$.
For each $\rho:=\left(\begin{smallmatrix}
a & b\\
c & d
\end{smallmatrix}\right)\in \SL_2(\ZZ)$ and each meromorphic function $f : \mathbb{H}\to \mathbb{C}$, the slash operator $f|_{w}\rho$ is defined by
\[\left(f|_{w}\rho\right)(z):=(cz+d)^{-w}f\left(\frac{az+b}{cz+d}\right). \]
Note that $(f|_{w}\rho_1)|_{w}\rho_2$ may not be equal to $f|_{w}\rho_1\rho_2$ for $\rho_1, \rho_2\in \SL_2(\ZZ)$.
For $f\in M_{w}(\Gamma_0(N),\chi)$ and a Dirichlet character $\phi$ modulo $M$, let
\[ f\otimes \phi(z) := \sum_{n=0}^{\infty} a_{f}(n)\phi(n)q^n \in M_{w}(\Gamma_0(NM^2),\chi \phi^2). \]
For convenience, if $n$ is not an integer, we let $a_{f}(n):=0$.
For each prime $p$, the Hecke operators $T_{p,k,\chi}$ on $M_{k}(\Gamma_0(N),\chi)$ and  $T_{p^2,k+\frac{1}{2},\chi}$ on $M_{k+\frac{1}{2}}(\Gamma_0(N),\chi)$ are defined by
\[ \left(\sum_{n=0}^{\infty} a_{f}(n)q^n\right)|T_{p,k,\chi}:=\sum_{n=0}^{\infty}\left(a_{f}(pn)+\chi(p)p^{k-1}a_{f}(n/p) \right) q^{n}\]
and
\[ \left(\sum_{n=0}^{\infty} a_{f}(n)q^n\right)|T_{p^2,k+\frac{1}{2},\chi}:=\sum_{n=0}^{\infty}\left(a_{f}(p^2n)+\chi(p)\left(\frac{(-1)^{k}n}{p}\right)p^{k-1}a_{f}(n)+ \chi(p^2)p^{2k-1}a_{f}(n/p^2) \right) q^{n}. \]
For $f\in M_{k}(\Gamma_0(N),\chi)$, if $f$ is an eigenform of $T_{p,k,\chi}$, we let $\lambda_f(p)$ denote the Hecke eigenvalue of $f$ at $p$.
Similarly, for $f\in M_{k+\frac{1}{2}}(\Gamma_0(N),\chi)$, if $f$ is an eigenform of $T_{p^2,k+\frac{1}{2},\chi}$, let $\lambda_{f}(p^2)$ denote the Hecke eigenvalue of $f$ at $p^2$.

Assume that $f$ is a cusp form of weight $k+\frac{1}{2}$ on $\Gamma_0(N)$ with nebentypus $\chi$.
For each positive square-free integer $t$, we define $A_{t}(n)$ by
\[ \sum_{n=1}^{\infty} \frac{A_t(n)}{n^s}:=\sum_{n=1}^{\infty} \frac{\chi(n)\left(\frac{(-1)^{k}t}{n}\right)}{n^{s-k+1}}\sum_{n=1}^{\infty}\frac{a_{f}(tn^2)}{n^s}.\]
Then, the Shimura Correspondence $\Sh_{t}$ is defined by
\[ \Sh_{t}(f)(z):= \sum_{n=1}^{\infty} A_{t}(n)q^{n}\in M_{2k}(\Gamma_0(N/2),\chi^2).\]
If $k>1$, then $\Sh_{t}(f)$ is a cusp form.
Note that if $p$ is a prime with $(p,N)=1$, then
\[ \Sh_{t}(f)|T_{p,2k,\chi^2}=\Sh_{t}\left(f|T_{p^2,k+\frac{1}{2},\chi}\right).\]
For details on the Shimura Correspondence, we refer to \cite{SH}.

Let $f\in S_{k+\frac{1}{2}}(\Gamma_0(N),\chi)$ have integral Fourier coefficients.
Let $M$ be an odd integer.
Assume that there is a positive integer $n_0$ such that $a_{f}(n_0)$ is relatively prime to $M$.
In the proofs of Theorem 1.1 and Lemma 2.1 in \cite{BO}, Bruinier and Ono obtained that
if there is a prime $p$ satisfying
\begin{equation*}
    p\equiv 1\pmod{NM}, \quad \left(\frac{n_0}{p}\right)=-1, \text{ and } f|T_{p^2,k+\frac{1}{2},\chi}\equiv 2f \pmod{M},
\end{equation*}
then for each integer $r$,
\begin{equation*}
    \#\{ n\leq X : a_{f}(n)\equiv r \pmod{M} \} \gg_{M} \begin{cases}
X \quad &\text{if }r=0,\\
\frac{\sqrt{X}}{\log X} \quad &\text{if }r\neq 0.
\end{cases}
\end{equation*}
This result, coupled with the Chinese Remainder Theorem, gives the following lemma.

\begin{lem}\label{lem : lem of distribution-half}
Let $M$ be an odd positive integer. Let $\ell$ be a prime with $\ell \mid M$, and $v_{\ell}$ be the maximum of positive integers $e$ such that $\ell^e \mid M$.
Let $a(n)$ be a sequence of integers with the following properties: 
\begin{enumerate}
    \item There exists a positive integer $n_0$ such that $(n_0a(n_0),M)=1$.
    \item For every $\ell\mid M$, there exists $f_{\ell} \in S_{k_{\ell}+\frac{1}{2}}(\Gamma_0(N\ell^2),\chi_{\ell})\cap \mathbb{Z}[[q]]$ with integer $k_{\ell}$ such that for all $n$,
\[ a(n)\equiv a_{f_{\ell}}(n) \pmod{\ell^{v_{\ell}}}.\]
\end{enumerate}

If there is a prime $p$ such that for every $\ell$,
\begin{equation}\label{e : cond of p-lemma 2.1}
    p\equiv 1 \pmod{2N\ell^{v_{\ell}+2} }, \quad \left(\frac{n_0}{p}\right)=-1, \text{ and } f_{\ell}|T_{p^2,k_{\ell}+\frac{1}{2},\chi_{\ell}}\equiv 2f_{\ell} \pmod{\ell^{v_{\ell}}},
\end{equation}
then for each integer $r$,
    \[ \#\{ n\leq X : a(n)\equiv r \pmod{M} \text{ and } (n,M)=1\} \gg_{M} \begin{cases}
X \quad &\text{if }r=0,\\
\frac{\sqrt{X}}{\log X} \quad &\text{if }r\neq 0.
\end{cases}\]
\end{lem}

\begin{rmk}
Let $\pi(X)$ be the number of primes less than or equal to $X$. For a subset $A$ of the set of primes, let  $\pi_A(X):=\#\{  p \leq X : \text{ primes } p \in A\}$. In this paper, ``a positive proportion of primes $p \in A$'' means that
\[
\liminf_{X\rightarrow \infty} \frac{\pi_A(X)}{\pi(X)}>0.
\]
\end{rmk}

\begin{proof}
The proof of this lemma follows those of Theorem 1.1 and Lemma 2.1 in \cite{BO}. For convenience, let $a_{f_{\ell}}(n):=0$ if $n$ is not an integer. Assume that $p$ satisfies \eqref{e : cond of p-lemma 2.1} for every $\ell$. Since $f_{\ell}|T_{p^2,k_{\ell}+\frac{1}{2},\chi_{\ell}}\equiv 2f_{\ell} \pmod{\ell^{v_{\ell}}} $, it follows that for each non-negative integer $e$,
\begin{multline*}
    a_{f_{\ell}}(n_0p^{e+2})+\chi_{\ell}(p)\left(\frac{(-1)^{k_{\ell}}n_0p^{e}}{p}\right)p^{k_{\ell}-1}a_{f_{\ell}}(n_0p^{e})+\chi_{\ell}(p^2)p^{2k_{\ell}-1}a_{f_{\ell}}(n_0p^{e-2})\\
    \equiv 2a_{f_{\ell}}(n_0p^{e})\pmod{\ell^{v_{\ell}}}.
\end{multline*}
Note that $p \equiv 1 \pmod{4}$ since $4 | N$. Thus, we have
\[ a_{f_{\ell}}\left(n_0 p^{e+2} \right)\equiv \begin{cases}
3a_{f_{\ell}}(n_0) &\pmod{\ell^{v_{\ell}}} \quad \text{if } e=0,\\
2 a_{f_{\ell}}\left(n_0p^e\right)-a_{f_{\ell}}\left(n_0p^{e-2}\right) &\pmod{\ell^{v_{\ell}}} \quad \text{if } e>0,
\end{cases}\]
and so
\[ a_{f_{\ell}}(n_0 p^{2e})\equiv (2e+1)a_{f_{\ell}}(n_0)\pmod{\ell^{v_{\ell}}}.\]
By the Chinese Remainder Theorem, we obtain
\begin{equation*}
    a(n_0p^{2e})\equiv (2e+1)a(n_0) \pmod{M}.
\end{equation*}
Note that $2a(n_0)$ is relatively prime to $M$. Thus, for each integer $r$, there is a positive integer $m_r$  such that $m_r$ is relatively prime to $M$ and
\[ a(m_r)\equiv r\pmod{M}.\]

Let $q$ be a prime such that for every $\ell\mid M$,
\begin{equation}\label{e : eq2}
    q\equiv -1 \pmod{2M N\ell^2m_r} \text{ and } f_{\ell}|T_{q^2,k_{\ell}+\frac{1}{2},\chi_{\ell}}\equiv 0 \pmod{M}.
\end{equation}
By the Shimura Correspondence and \cite[\S6.4]{SE}, we have that the proportion of $q$ satisfying \eqref{e : eq2} in the set of primes is positive.
For a fixed $q$, if $n$ is a positive integer with $(n,q)=1$, then
\begin{equation*}
    \begin{aligned}
    a_{f_{\ell}}\left(nq^3\right)&\equiv -\chi_{\ell}(q)\left(\frac{(-1)^{k_{\ell}}nq}{q}\right)q^{k_{\ell}-1}a_{f_{\ell}}(nq)-\chi_{\ell}(q^2)q^{2k_{\ell}-1}a_{f_{\ell}}\left(\frac{n}{q}\right)\\
    &\equiv 0 \pmod{\ell^{v_{\ell}}}.
    \end{aligned}
\end{equation*}
Thus, we obtain the proof of Lemma \ref{lem : lem of distribution-half} for the case when $r\equiv 0 \pmod{M}$.

Assume that $r \not \equiv 0\pmod{M}$.
Then, we have
\[ a_{f_{\ell}}(q^2m_r)+\chi_{\ell}(q)\left(\frac{(-1)^{k_{\ell}}m_r}{q}\right)q^{k_{\ell}-1}a_{f_{\ell}}(m_r)\equiv 0 \pmod{\ell^{v_{\ell}}}. \]
This implies that
\[ a_{f_{\ell}}(q^2m_r)\equiv \chi_{\ell}(-1)\left(\frac{(-1)^{k_{\ell}}}{-1}\right)(-1)^{k_{\ell}}a_{f_{\ell}}(m_r)\equiv r\pmod{\ell^{v_{\ell}}}. \]
By the Chinese Remainder Theorem, we obtain
\[ a(q^2m_r)\equiv  r\pmod{M}. \]
Since there is a positive proportion of primes $q$ satisfying \eqref{e : eq2} for every prime $\ell |M$,
the prime number theorem implies that
\[ \#\{n\leq X : a(n)\equiv r \pmod{M} \text{ and } (n,M)=1\}\gg_{M}\frac{\sqrt{X}}{\log X}. \]
\end{proof}

\section{Proofs of Theorems \ref{thm : thm 1} and \ref{cor : cor 1}}\label{s : distribution of the Fourier coefficients}
In this section, we study the distribution of the Fourier coefficients of modular forms of half-integral weight with integral coefficients modulo an integer, and we prove Theorems \ref{thm : thm 1} and \ref{cor : cor 1}.

Let $k$ and $N$ be integers with $4\mid N$.
Let $M$ be an odd integer. For each prime divisor $\ell$ of $M$, let $v_{\ell}$ be the maximum of positive integers $e$ such that $\ell^e \mid M$.
Let $f$ be a modular form in $M_{k+\frac{1}{2}}(\Gamma_0(N),\chi)$ with integral coefficients. To prove Theorem \ref{thm : thm 1}, for each prime divisor $\ell$ of $M$, we construct a cusp form $f_{\ell}(z)$ of half-integral weight on $\Gamma_0(N\ell^2)$ with nebentypus $\chi_{\ell}$ such that
\[ f_{\ell}(z)\equiv \sum_{\ell\nmid n} a_{f}(n)q^n \pmod{\ell^{v_{\ell}}}. \]
Let $n_0$ be a square-free positive integer such that $n_0 \nmid N$ and that $a_f(n_0m^2) \neq 0$ for some non-zero $m \in \mathbb{Z}$. Then, by Lemma \ref{lem : lem of distribution-half}, the proof of Theorem \ref{thm : thm 1} reduces to proving that for each odd integer $M$ relatively prime to $Nn_0ma_{f}(n_0m^2)$, there exists a prime $p$ satisfying \eqref{e : cond of p-lemma 2.1} in Lemma \ref{lem : lem of distribution-half}.
The following lemma gives the existence of such a prime $p$.

\begin{lem}\label{lem : main lem 1}
Let $M$ be an odd integer.
Assume that for each prime  $\ell \mid M$, $f_{\ell} \in S_{k_{\ell}+\frac{1}{2}}(\Gamma_0(N\ell^2),\chi_{\ell})$ is a non-zero function with integral coefficients and that $k_{\ell}\geq 2$.
If $n_0$ is a square-free integer such that
\[ n_0\nmid MN, \]
then there is a positive proportion of primes $p$ satisfying
\begin{equation*}
    p\equiv 1\pmod{2MN\ell^2}, \quad \left(\frac{n_0}{p}\right)=-1, \text{ and }f_{\ell}|T_{p^2,k_{\ell}+\frac{1}{2},\chi_{\ell}}\equiv 2f_{\ell} \pmod{\ell^{v_{\ell}}}
\end{equation*}
for every prime divisor $\ell$ of $M$.
\end{lem}
To prove Lemma \ref{lem : main lem 1}, we use the theory of $\ell$-adic Galois representations attached to newforms. Let $F$ be a newform in $S_{2k}(\Gamma_0(N),\chi)$ such that the first Fourier coefficient of $F$ is one. Let $K$ be a number field containing an $N$-th root of unity $\xi_{N}$ such that the Fourier coefficients of $F$ are in $K$.
Let $G_{\Q}$ be the absolute Galois group of $\Q$.
For a prime $\ell$, Deligne \cite{D} constructed an odd Galois representation $\rho_{F} : G_{\Q} \to \GL_2\left(K\otimes \mathbb{Q}_{\ell}\right)$ attached to $F$
such that $\rho_{F}$ is unramified at $p$ with $p\nmid N\ell$, and such that for $p$ with $p\nmid N\ell$,
\begin{equation}\label{e : eq of Galois repre, tr det}
    \tr(\rho_{F}(\Frob_p))=\lambda_{F}(p) \text{ and }  \det(\rho_{F}(\Frob_p))=\chi(p)p^{2k-1}.
\end{equation}
Here, $\Frob_p$ is a Frobenius element at $p$.
Let $\mathcal{O}_{K}$ be the ring of integers of $K$.
Since $G_{\Q}$ is compact, we take an $\mathcal{O}_{K}\otimes \mathbb{Z}_{\ell}$-lattice $\Lambda$ of rank $2$ such that $\rho_{F}(g)(\Lambda)=\Lambda$ for any $g\in G_{\Q}$.
From this, we deduce that there is a Galois representation $\rho_{F} : G_{\Q} \to \GL_2\left(\mathcal{O}_{K}\otimes \mathbb{Z}_{\ell}\right)$ satisfying \eqref{e : eq of Galois repre, tr det}.

Assume that $G\in S_{2k}(\Gamma_0(N),\chi)$ is a Hecke eigenform of $T_{p,2k,\chi}$ for all $p$ with $p\nmid N$ and that the Fourier coefficients of $G$ are in $\mathcal{O}_{K}$.
Then, we have
\[ G(z)=\sum_{\delta\mid N} c_{\delta} \cdot G_{\delta}(\delta z), \]
where $c_{\delta}\in \CC$ and $G_{\delta}$ is a newform in $S_{2k}(\Gamma_0(N_{\delta}),\chi)$ with $N_{\delta}\mid N$.
If $c_{\delta}\neq 0$, then $\lambda_{G}(p)=\lambda_{G_{\delta}}(p)$ for all $p$ with $p\nmid N$.
Thus, there is a Galois representation $\rho_{G}:G_{\Q}\to \GL_2\left(\mathcal{O}_{K}\otimes \mathbb{Z}_{\ell}\right)$ such that for a prime $p$ with $p\nmid N\ell$, $\rho_{G}$ is unramified at $p$ and
\[ \tr(\rho_{G}(\Frob_p))=\lambda_{G}(p) \text{ and }  \det(\rho_{G}(\Frob_p))=\chi(p)p^{2k-1}. \]
With these notations, we prove Lemma \ref{lem : main lem 1}.
\begin{proof}[Proof of Lemma \ref{lem : main lem 1}]
For a prime $\ell \mid M$, let $t_{\ell}$ be a positive square-free integer such that  $\Sh_{t_{\ell}}(f_{\ell})\neq 0$, and $A_{\ell}$ be a constant such that all the Fourier coefficients of $A_{\ell}\cdot \Sh_{t_{\ell}}(f_{\ell})$ are algebraic integers and 
\[ A_{\ell}\cdot \Sh_{t_{\ell}}(f_{\ell})\not\equiv 0 \pmod{\ell}. \]
Let $F_{\ell}:=A_{\ell}\cdot \Sh_{t_{\ell}}(f_{\ell})$.
Then, $F_{\ell} \in S_{2k_{\ell}}(\Gamma_0(N\ell^2/2),\chi^2_{\ell})$.
Let $\mathcal{B}_{\ell}$ be a basis of $S_{2k_{\ell}}(\Gamma_0(N\ell^2/2),\chi^2_{\ell})$ consisting of eigenforms $G$ of $T_{p,2k_{\ell},\chi_{\ell}^2}$ for all $p$ with $(p,N\ell^2/2)=1$ such that the Fourier coefficients of $G$ are algebraic integers. Note that the field generated over the rational number
field $\mathbb{Q}$ by $a_G(n)$ for all $n \geq 1$ is of finite degree over $\mathbb{Q}$. Thus, let $K$ be a number field, containing an $N\ell^2$-th root $\xi_{N\ell^2}$ of unity, such that if $G \in \mathcal{B}_{\ell}$ for some prime $\ell \mid M$, then the Fourier coefficients of $G$ are in $K$.
Then, we have
\[ F_{\ell}=\sum_{G\in \mathcal{B}_{\ell}} c_{G} \cdot G.\]
for some $c_G \in K$. Let $\mathcal{O}_{K}$ be the ring of integers of $K$. Then, for each prime $\ell|M$, there exists a non-negative integer $u(\ell)$ such that for every $G \in \mathcal{B}_{\ell}$, the valuations of $\ell^{u(\ell)} c_G$ at any places lying over $\ell$ are non-negative.

For $x \in \mathcal{O}_{K}\otimes \ZZ_{\ell}$, we denote by $\overline{x}$ the reduction of $x$ in $\mathcal{O}_{K}\otimes \left(\mathbb{Z}_{\ell}/\ell^{v_{\ell}+u(\ell)}\mathbb{Z}_{\ell}\right)$.
For each $G\in \mathcal{B}_{\ell}$, there exists an odd Galois representation $\overline{\rho_{G}}: G_{\Q}\to \GL_2\left(\mathcal{O}_{K}\otimes \mathbb{Z}_{\ell}/\ell^{v_{\ell}+u(\ell)}\mathbb{Z}_{\ell}\right)$ such that for $p$ with $p\nmid N\ell$,
\[ \tr(\overline{\rho_{G}}(\Frob_p))= \overline{\lambda_G(p)} \;\text{ and } \; \det(\overline{\rho_{G}}(\Frob_p))= \overline{\chi^2_{\ell}(p)p^{2k_{\ell}-1}}
\]
in $\mathcal{O}_{K}\otimes \left(\mathbb{Z}_{\ell}/\ell^{v_{\ell}+u(\ell)}\mathbb{Z}_{\ell}\right)$.
We define
\[ \rho_1 : G_{\Q}\to \prod_{\ell\mid M} \prod_{G\in \mathcal{B}_{\ell}} \GL_2\left(\mathcal{O}_{K}\otimes \ZZ_{\ell}/\ell^{v_{\ell}+u(\ell)}\ZZ_{\ell}\right) \]
by
\[
\rho_1:=\prod_{\ell\mid M} \prod_{G\in \mathcal{B}_{\ell}} \overline{\rho_G}.
\]
Let $E_1$ be the subfield of $\overline{\Q}$ fixed by all elements in the kernel of $\rho_1$.
Then, for the identity element $Id$ in $G_{\mathbb{Q}}$, we have
\[\rho_1(Id)=(I_2,I_2,\dots, I_2).\]
Here, $I_2$ denotes the identity matrix of size $2$.

Let $N_M$ be a positive integer defined by
\[N_{M}:=2MN\prod_{\ell\mid M}\ell^2. \]
We define
\[\rho_2:G_{\Q}\to \GL_1\left(\ZZ/N_M\ZZ\right)\]
by $g(\xi_{N_M})=\xi_{N_M}^{\rho_2(g)}$ for $g\in G_{\Q}$. Let $E_2$ be the subfield of $\overline{\Q}$ fixed by all elements in the kernel of $\rho_2$.
Note that if $p\nmid N_M$, then $\rho_2(\Frob_p)\equiv p \pmod{N_M}$.
Thus, $p \equiv 1 \pmod{N_M}$ if and only if $\Frob_p|_{\mathbb{Q}(\xi_{N_M})}=Id|_{\mathbb{Q}(\xi_{N_M})}$.

We define $\rho_3 : G_{\Q}\to \GL_1(\ZZ)$ by $\rho_3(g):=\frac{g(\sqrt{n_0})}{\sqrt{n_0}}$ for $g\in G_{\Q}$. Let $E_3$ be the subfield of $\overline{\Q}$ fixed by all elements in the kernel of $\rho_3$. Then, for every prime $p\nmid 2n_0$, $$\rho_3(\Frob_p)=\left(\frac{n_0}{p}\right).$$
Let $\sigma \in \Gal(\mathbb{Q}(\sqrt{n_0})/\mathbb{Q})$ be a nontrivial automorphism.

We claim that there exists $g \in G_{\mathbb{Q}}$ such that
\begin{enumerate}
    \item $g|_{E_1}=Id|_{E_1}$,
    \item $g|_{E_2}=Id|_{E_2}$,
    \item $g|_{E_3}=\sigma$.
\end{enumerate}
Let $E$ be the composite field of $E_1$, $E_2$ and $E_3$.
Then, by the Chebotarev density theorem (for details, see \cite[Theorem 13.4]{NEU}), there is a positive proportion of primes $p$ such that
$\Frob_p|_{E}=g|_{E}$. The condition (1) implies that for all $G \in \mathcal{B}_{\ell}$, \[ \lambda_{G}(p)\equiv 2\pmod{\ell^{v_{\ell}+u(\ell)}}, \]
and so
\[ F_{\ell}|T_{p,2k_{\ell},\chi_{\ell}^2}-2F_{\ell}\equiv \sum_{G\in \mathcal{B}_{\ell}} c_{G}\left(\lambda_{G}(p)-2\right)G\equiv 0 \pmod{\ell^{v_{\ell}}}. \]
Since the Hecke operators commute with the Shimura Correspondence, we have
\[ f_{\ell}|T_{p^2,k_{\ell}+\frac{1}{2},\chi_{\ell}}\equiv 2 f_{\ell} \pmod{\ell^{v_{\ell}}}. \]
The condition (2) implies that $p \equiv 1 \pmod{N_M}$. The condition (3) implies that for $p$ with $p\nmid 2n_0$,
\[ \left(\frac{n_0}{p}\right)=\rho_3(\Frob_p)=\frac{\sigma(\sqrt{n_0})}{\sqrt{n_0}}=-1.\]
Thus, we obtain the proof of Lemma \ref{lem : main lem 1}.

To complete the proof of Lemma \ref{lem : main lem 1}, we prove the claim. The claim is equivalent to the fact that there exists $g' \in \Gal(E \slash \mathbb{Q})$ such that $g'|_{E_1E_2}=Id|_{E_1E_2}$ and $g'|_{E_3}=\sigma$.
Assume that a prime $p$ is relatively prime to $NM$.
Then, for each $G$ in $\mathcal{B}_{\ell}$, the Galois representation $\overline{\rho_{G}}$ is unramified at $p$.
By the definition of $\rho_1$, we deduce that $\rho_1$ is unramified at $p$.
Since $p\nmid N_{M}$, it follows that $\rho_2$ is unramified at $p$.
Thus, $p$ is unramified in $E_1E_2$.
Note that $E_3=\mathbb{Q}(\sqrt{n_0})$.
Since $n_0\nmid NM$, there is a prime divisor $q$ of $n_0$ such that $q$ is relatively prime to $NM$.
Then, $q$ is unramified in $E_1E_2$ and is ramified in $E_3$.
This implies that $E_1E_2$ does not contain $E_3$.
Thus, $E_3\cap E_1E_2$ is a proper subfield of $E_3$ and so
\[ E_3\cap E_1E_2=\Q. \]

Since $E_1E_2/\Q$ and $E_3/\Q$ are Galois extensions, $E/\Q$ is a Galois extension.
Note that the intersection of $E_3$ and $E_1E_2$ is $\Q$.
By \cite[Chapter 14. Corollary 22]{DS}, we have
\[ \Gal(E/\Q)\cong \Gal(E_1E_2/\Q)\times \Gal(E_3/\Q). \]
Thus, there is $g'\in \Gal(E/\Q)$ such that
$g'|_{E_1E_2}=Id|_{E_1E_2}$ and that $g'|_{E_3}=\sigma$. This completes the proof of the claim.

\end{proof}

We now prove Theorem \ref{thm : thm 1} by using Lemmas \ref{lem : lem of distribution-half} and \ref{lem : main lem 1}.
\begin{proof}[Proof of Theorem \ref{thm : thm 1}]
For each odd prime $p$, let $\chi_{p}^{triv}$ be the trivial character modulo $p$ and
\begin{equation}\label{e : def of F_p}
    F_{p}(z):=\begin{cases}
\frac{\eta^{p^2}(z)}{\eta(p^2z)} \quad &\text{if }p\geq 5, \\
\frac{\eta^{27}(z)}{\eta^3(9z)} \quad &\text{if }p=3.
\end{cases}
\end{equation}
Let $M$ be a positive integer with $M=\prod_{\text{prime }\ell\mid M}\ell^{v_{\ell}}$.
For each prime divisor $\ell$ of $M$, let
\[ f_{\ell}:= (f\otimes \chi_{\ell}^{triv})F_{\ell}^{\ell^{v_{\ell}-1}}.\]
We claim that
\[f_{\ell}(z)\in S_{k_{\ell}+\frac{1}{2}}(\Gamma_0(N\ell^2),\chi)\cap \mathbb{Z}[[q]]\]
such that
\[ f_{\ell}(z)\equiv \sum_{\ell\nmid n} a_{f}(n)q^{n} \pmod{\ell^{v_{\ell}}}.\]
Here,
\[ k_{\ell}:=k+C_{\ell}\cdot\frac{\ell^{v_{\ell}-1}(\ell^2-1)}{2}, \]
where
\[ C_{\ell}:=\begin{cases}
    1  \quad &\text{if $\ell\geq 5$}, \\
    3 \quad &\text{if $\ell=3$}.
\end{cases} \]

Note that $M$ is relatively prime to $Nn_0ma_{f}(n_0m^2)$ and that $n_0\nmid N$.
Thus, we have $n_0\nmid NM$.
Then, Lemma \ref{lem : main lem 1} implies that there is a positive proportion of primes $p$ such that
\[ p\equiv 1\pmod{2MN\ell^2}, \quad \left(\frac{n_0}{p}\right)=-1, \text{ and } f_{\ell}|T_{p^2,k_{\ell}+\frac{1}{2},\chi}\equiv 2f_{\ell} \pmod{\ell^{v_{\ell}}} \]
for every prime divisor $\ell$ of $M$.
Therefore, Lemma \ref{lem : lem of distribution-half} implies that for each integer $r$,
    \[ \#\{ n\leq X : a_{f}(n)\equiv r \pmod{M}\} \gg_{M} \begin{cases}
X \quad &\text{if }r=0,\\
\frac{\sqrt{X}}{\log X} \quad &\text{if }r\neq 0.
\end{cases}\]

To complete the proof of Theorem \ref{thm : thm 1}, we prove the claim.
We follow the proof of Theorem 3.1 in \cite{TR}.
Note that $F_{\ell}$ is a modular form of weight $C_{\ell}\cdot\frac{\ell^2-1}{2}$ on $\Gamma_0(\ell^2)$ and that $F_{\ell}\equiv 1\pmod{\ell}$.
Then, $f_{\ell}$ is a modular form of  weight $k_{\ell}$ on $\Gamma_0(N\ell^2)$ with nebentypus $\chi$ such that
\[ f_{\ell}(z)\equiv \sum_{\ell\nmid n}a_{f}(n)q^n \pmod{\ell^{v_{\ell}}}. \]
Let $\rho:=\left(\begin{smallmatrix}
a & b\\
\ell^2 c & d
\end{smallmatrix}\right)$ be a matrix in $\SL_2(\ZZ)$.
Since $\ell\nmid N$, we obtain that $\left(f\otimes \chi^{triv}_{\ell}\right)|_{k+\frac{1}{2}}\rho$ has the Fourier expansion of the form
\[ \left(f\otimes\chi^{triv}_{\ell}\right)|_{k+\frac{1}{2}}\rho=\sum_{\ell\nmid n} a_{f}\left(\rho : \frac{n}{w_{\rho}}\right)q^{\frac{n}{w_{\rho}}}, \]
where $w_{\rho}$ is a positive integer (see, (2.5) and (3.12) in \cite{TR}).
This implies that $f\otimes \chi^{triv}_{\ell}$ vanishes at every cusp $\frac{a}{\ell^2 c}\in \Q$.
Since $F_{\ell}$ vanishes at every cusp $\frac{a}{c}\in \Q$ with $\ell^2\nmid c$, we conclude that $\left(f\otimes \chi^{triv}_{\ell}\right)F_{\ell}$ vanishes at every cusp $s\in \Q$.
Therefore, we complete the proof of the claim.
\end{proof}

As an application of Theorem \ref{thm : thm 1}, we consider an analogue of Newman's Conjecture for $t$-core partitions (see Section \ref{sss : t-core}).
From the generating function for $c_t(n)$, we have
\[ \frac{\eta^{t}(tz)}{\eta(z)}=\sum_{n=0}^{\infty} c_{t}(n)q^{n+\delta_{t}} \quad \left(\delta_{t}:=\frac{t^2-1}{24}\right). \]
For convenience, let $c_t(s):=0$ if $s$ is not a non-negative integer.
Then, we obtain
\begin{equation}\label{e : generating function of t-core}
    \frac{\eta^{t}(24tz)}{\eta(24z)}=\sum_{n=0}^{\infty} c_t\left(\frac{n+1-t^2}{24} \right) q^{n}\in M_{\frac{t-1}{2}}(\Gamma_0(576t),\chi)
\end{equation}
with real nebentypus $\chi$.
The weight of $\frac{\eta^{t}(24tz)}{\eta(24z)}$ is integral for odd integer $t$. Thus, let us note the following lemma about the distribution of the Fourier coefficients of cusp forms of integral weight modulo an integer $M$, which is given by following the argument of Serre in \cite[\S6.4 and \S6.5]{SE}.

\begin{lem}\label{lem : lem of distribution_integral}
Let $M$, $\ell$ and $v_{\ell}$ be as in Lemma \ref{lem : lem of distribution-half}.
Let $\{a(n)\}$ be a sequence of integers such that there exists a positive integer $n_0$ such that $(n_0a(n_0),M)=1$. Assume that for each $\ell$, there exists $f_{\ell} \in S_{k_{\ell}}(\Gamma_0(N\ell^2),\chi_{\ell})\cap \mathbb{Z}[[q]]$ with integer $k_{\ell}$ such that for all $n$,
\[ a(n)\equiv a_{f_{\ell}}(n) \pmod{\ell^{v_{\ell}}}.\]
Let $c$ be a positive integer.
Then, for each integer $r$,
\[ \#\{ n\leq X : a(n)\equiv r \pmod{M} \text{ and } (n,M)=1 \} \gg_{M,c} \begin{cases}
    X \quad &\text{if }r=0,\\
    \frac{X(\log \log X)^{c}}{\log X} \quad &\text{if }r\neq 0.
    \end{cases}\]
\end{lem}
\begin{proof}
Following the argument in \cite[\S6.4]{SE}, we see that there is a positive proportion of primes $p$ such that for every prime divisor $\ell$ of $M$,
\begin{equation}\label{e : cond dist_inte}
    p\equiv 1\pmod{MN\ell^2} \text{ and } f_{\ell}|T_{p,k_{\ell},\chi_{\ell}}\equiv 2f_{\ell}\pmod{M}.
\end{equation}
Assume that $p$ satisfies \eqref{e : cond dist_inte} and $(p,n_0)=1$.
Since $f_{\ell}|T_{p,k_{\ell},\chi_{\ell}}\equiv 2f_{\ell} \pmod{\ell^{v_{\ell}}}$, it follows that for each non-negative integer $e$,
\[ a_{f_{\ell}}(n_0p^{e+1})+\chi_{\ell}(p)p^{k_{\ell}-1}a_{f_{\ell}}(n_0p^{e-1})\equiv 2 a_{f_{\ell}}(n_0p^{e}) \pmod{\ell^{v_{\ell}}}. \]
Note that $p\equiv 1\pmod{MN\ell^2}$ and $(p,n_0)=1$.
Thus, we obtain
\[ a_{f_{\ell}}(n_0p^{e+1})\equiv \begin{cases}
2a_{f_{\ell}}(n_0) &\pmod{\ell^{v_{\ell}}} \quad \text{if }e=0,\\
2a_{f_{\ell}}(n_0p^{e})-a_{f_{\ell}}(n_0p^{e-1}) &\pmod{\ell^{v_{\ell}}} \quad \text{if }e>0.
\end{cases}\]
From this, we have
\[ a_{f_{\ell}}(n_0p^{e})\equiv (e+1)a_{f_{\ell}}(n_0)\pmod{\ell^{v_{\ell}}}. \]
By the Chinese Remainder Theorem, we obtain
\begin{equation}\label{eq 5}
    a(n_0p^{e})\equiv (e+1)a(n_0) \pmod{M}. 
\end{equation}
The remaining part of the proof of Lemma \ref{lem : lem of distribution_integral} is similar to \cite[\S6.5]{SE} as follows.

Since $(a(n_0),M)=1$, \eqref{eq 5} implies that for each integer $r$ with $M\nmid r$, there is a positive integer $m_r$ such that 
\[ a(m_r)\equiv 2^{-c-1}r\pmod{M}. \]
Assume that $\{p_1,\dots, p_{c+1}\}$ is a set of primes relatively prime to $m_r$ such that each $p_i$ satisfies \eqref{e : cond dist_inte}.
By \eqref{eq 5}, we have 
\[ a(m_r p_1 \dots p_{c+1})\equiv 2^{c+1}\cdot 2^{-c-1}r\equiv r \pmod{M}. \]
Landau \cite{LA} proved that the number of integers less than $X$ that are the product of distinct $s+1$ primes is asymptotically equivalent to $\frac{X(\log \log X)^{s}}{s! \log X}$ as $X$ goes to infinity.  
Since a positive proportion of primes $p$ satisfy $(p,m_r)=1$ and \eqref{e : cond dist_inte}, we obtain by the Landau's result \cite{LA} that if $r\not\equiv 0\pmod{M}$, then  
\[ \#\{n\leq X : a(n)\equiv r \pmod{M} \text{ and } (n,M)=1 \}\gg_{M,c} \frac{X(\log \log X)^{c}}{\log X}. \]
Thus, to complete the proof, we assume that $r\equiv 0\pmod{M}$.

By the result of Serre \cite[\S 6.4]{SE}, there is a prime $p$ such that 
\[ f_{\ell}|T_{p,k_{\ell},\chi_{\ell}}\equiv 0\pmod{M}. \]
For any positive integers $n$ with $(n,p)=1$, we have 
\begin{equation*}
    \begin{aligned}
        a_{f_{\ell}}(pn)+\chi_{\ell}(p)p^{k_{\ell}-1}a_{f_{\ell}}(n/p)\equiv a_{f_{\ell}}(pn)\equiv 0\pmod{M}.
    \end{aligned}
\end{equation*}
This implies that if $(n,p)=1$, then $a(pn)\equiv 0\pmod{M}$. 
Therefore, we complete the proof.

\end{proof}
We now prove Theorem \ref{cor : cor 1} by using Theorem \ref{thm : thm 1} and Lemma \ref{lem : lem of distribution_integral}.

\begin{proof}[Proof of Theorem \ref{cor : cor 1}]
Let $a_t(n)$ be the $n$-th Fourier coefficient of $\frac{\eta^{t}(24tz)}{\eta(24z)}$.
Let us recall from \eqref{e : generating function of t-core} that
\[ a_t(n)=c_t\left(\frac{n+1-t^2}{24} \right). \]
Assume that $t$ is an even integer larger than $2$.
By the definition of $t$-core partition, we have $a_t(t^2-1)=c_t(0)=1$ and $a_t(t^2+23)=c_t(1)=1$.
Let $n_0$ and $n_1$ be square-free integers such that $n_0m_0^2=t^2-1$ and $n_1m_1^2=t^2+23$ for some integers $m_0$ and $m_1$.

To get a contradiction, we assume that $n_0$ and $n_1$ both divide $576t$.
Since $(2t,n_0)=(2t,n_1)=1$, it follows that $n_0\mid 3$ and $n_1\mid 3$.
Note that $t^2-1\equiv t^2+23\equiv 3\pmod{4}$.
Thus, we have $n_0=n_1=3$.
From this, we obtain $m_1^2-m_0^2=8$ and so $(m_1^2,m_0^2)=(9,1)$ and $t=2$.
This contradicts the hypothesis that $t>2$.
Hence, at least one of $n_0$ and $n_1$ does not divide $576t$.
Recall that
$$\frac{\eta^{t}(24tz)}{\eta(24z)}\in M_{\frac{t-1}{2}}\left(\Gamma_0(576t),\chi \right).$$
Thus, $\frac{\eta^{t}(24tz)}{\eta(24z)}$ is not a linear combination of single-variable theta series since $n_0\nmid 576t$ or $n_1\nmid 576t$.
Therefore, Theorem \ref{thm : thm 1} implies that if $M$ is relatively prime to $3t(t^2-1)(t^2+23)$, (i.e., $M$ is relatively prime to $576t(t^2-1)(t^2+23)$), then for each integer $r$,
\[\#\{ n\leq X : c_{t}(n)\equiv r \pmod{M} \} \gg_{M} \begin{cases}
X \quad &\text{if }r=0,\\
\frac{\sqrt{X}}{\log X} \quad &\text{if }r\neq 0.
\end{cases}\]

We assume that $t$ is an odd integer, and a positive integer $M$ is relatively prime to $3t(t^2-1)$.
For each prime divisor $\ell$ of $M$, let $v_{\ell}$ be the maximum of positive integers $e$ such that $\ell^e \mid M$.
Let $F_{\ell}$ be as in \eqref{e : def of F_p} and
\[ f_{\ell}:=\left(\frac{\eta^{t}(24tz)}{\eta(24z)}\otimes \chi_{\ell}^{triv} \right)F_{\ell}^{\ell^{v_{\ell}-1}}. \]
Note that $(\ell,576t)=1$. By following the proof of Theorem \ref{thm : thm 1}, we deduce that $f_{\ell}$ is a cusp form of integral weight on $\Gamma_0(576t\ell^2)$ with nebentypus $\chi$, and
\[ f_{\ell}(z)\equiv \sum_{\ell\nmid n} a_{t}(n)q^n \pmod{\ell^{v_{\ell}}}.\]
Since $(M,t^2-1)=1$, we have the following: for every prime divisor $\ell$ of $M$,
\[ a_{f_{\ell}}(t^2-1)\equiv a_{t}(t^2-1) \pmod{\ell^{v_{\ell}}}. \]
Note that $a_{t}(t^2-1)=c_t(0)=1$, and that $M$ is relatively prime to $576t(t^2-1)$ if and only if $M$ is relatively prime to $3t(t^2-1)$.  Therefore, Lemma \ref{lem : lem of distribution_integral} shows that for each integer $r$,
 \[ \#\{ n\leq X : c_t(n)\equiv r \pmod{M} \} \gg_{M,c} \begin{cases}
    X \quad &\text{if }r=0,\\
    \frac{X(\log \log X)^{c}}{\log X} \quad &\text{if }r\neq 0,
    \end{cases}\]
where $c$ is a positive integer.
\end{proof}

\section{Proofs of Theorems \ref{thm : thm 2}, \ref{thm : thm of weak}, and \ref{thm : thm 3}}\label{s : section of weak}
In this section, we study the distribution of the Fourier coefficients of weakly holomorphic modular forms with integral coefficients  modulo an integer $M$, and we prove Theorems \ref{thm : thm 2}, \ref{thm : thm of weak}, and \ref{thm : thm 3}.
Note that Theorems \ref{thm : thm 2} and \ref{thm : thm 3} are implied by Theorem \ref{thm : thm of weak}.

Let us recall some notation defined in the previous sections. 
For a positive integer $d$, $B_{d}$ denotes the set of positive integers $M$ such that the number of prime divisors of $M$ is $d$.
Let $k$ be an integer and $N$ be a positive integer.
If $k$ is an odd integer, then assume that $4\mid N$.
For each weakly holomorphic modular form $f$ of weight $\frac{k}{2}$ on $\Gamma_0(N)$ with nebentypus $\chi$ and  $\rho\in \SL_2(\mathbb{Z})$, $f|_{\frac{k}{2}}\rho$ having the form
\[\left(f|_{\frac{k}{2}} \rho\right) (z) = \sum_{n\gg -\infty} a_{f}\left(\rho : \frac{n}{w_{\rho}} \right) q^{\frac{n}{w_{\rho}}}  \]
for some positive integer $w_{\rho}$.
We define $\Omega(f)$ as the subset of $\mathbb{Q}^{\times}$ consisting of negative $\frac{n}{w_{\rho}}$ for which $a_{f}\left(\rho:\frac{n}{w_{\rho}}\right)\neq 0$ for some $\rho\in \SL_2(\ZZ)$, and $\overline{\Omega(f)}$ denotes the projection of $\Omega(f)$ onto $\Q^{\times}/\Q^{\times^2}$.

Before proving Theorem \ref{thm : thm of weak}, let us give the result of Treneer \cite{TR} on the construction of a cusp form which is congruent to a linear combination of quadratic twists of $f$ modulo a power of a prime.

\begin{thm}[Theorem 3.1 in \cite{TR}]\label{thm : thm of Tr}
Suppose that $\ell$ is an odd prime, and that $k$ is an integer. Let $N$ be a positive
integer with $4 |N$ and $(N, \ell) = 1$. Let $K$ be a number field with ring of integers $\mathcal{O}_K$, and
suppose that $f$ is a weakly holomorphic modular form of weight $\frac{k}{2}$ on $\Gamma_1(N)$ with coefficients in $\mathcal{O}_K$. Let $\{r_1,\dots, r_{u}\}$ be a set of representatives of $\overline{\Omega(f)}$ such that every $r_i$ is an integer.
Assume that $\left(\ell,N \prod_{i=1}^{u}r_i \right)=1$. If $(\frac{r_i}{\ell})$ are the same for all $r_i$,
then for every positive integer $e$, there is an integer $\beta \geq e-1$ and a cusp form
\[ f_{\ell,e}\in S_{\frac{k+\ell^{\beta}(\ell^2-1)}{2}}(\Gamma_1(N\ell^2))\cap \mathcal{O}_K[[q]] \]
with the property that
\[ f_{\ell,e}(z)\equiv \sum_{\left(\frac{r_1n}{\ell}\right)=-1} a_{f}(n)q^n \pmod{\ell^e}. \]
\end{thm}

\begin{rmk}
In \cite{TR}, Treneer constructed a cusp form
\[ f_{\ell,e}:=\frac{1}{2}\left(f\otimes \chi_{\ell}^{triv}-\left(\frac{r_1}{\ell}\right)f\otimes \left(\frac{\cdot}{\ell}\right)\right)F_{\ell}^{\ell^{\beta}} \]
for some sufficiently large integer $\beta\geq e-1$.
Here, $F_{\ell}$ is defined as in \eqref{e : def of F_p}.
Note that if $f$ is a weakly holomorphic modular form of weight $\frac{k}{2}$ on $\Gamma_0(N)$ with nebentypus $\chi$, then $f_{\ell,e}$ is a cusp form of weight $\frac{k+\ell^{\beta}(\ell^2-1)}{2}$ on $\Gamma_0(N\ell^2)$ with nebentypus $\chi$.
\end{rmk}

Let $M:=\prod_{\text{prime }\ell\mid M}\ell^{v_{\ell}}$ be an odd integer with $(M,N)=1$. To prove Theorem  \ref{thm : thm of weak}, we apply Lemmas \ref{lem : lem of distribution-half}, \ref{lem : main lem 1}, and \ref{lem : lem of distribution_integral} to $f_{\ell,v_{\ell}}$ with the following lemma.

\begin{lem}\label{lem : density of k-primes}
Let $P$ be the set of primes, and $A$ be a subset of $P$.
For a positive integer $m$, let
\[\pi_{m,A}(X):=\#\left\{\prod_{i=1}^{m} p_i\leq X:  p_i \text{ are distinct primes in } A \right\}. \]
Assume that
\[ \lim_{X\to \infty} \frac{\pi_{1,A}(X)}{\pi_{1,P}(X)}=\mu \]
for some $\mu$ with $0\leq \mu\leq 1$.
Then, for each positive integer $m$,
\begin{equation*}
     \lim_{X\to \infty} \frac{\pi_{m,A}(X)}{\pi_{m,P}(X)}=\mu^{m}.
\end{equation*}
\end{lem}
\begin{proof}
We prove Lemma \ref{lem : density of k-primes} by mathematical induction on $m$.
When $m=1$, Lemma \ref{lem : density of k-primes} holds by the assumption.
Now, we assume that Lemma \ref{lem : density of k-primes} holds for all positive integers $m$ less than $k_0+1$, where $k_0$ is a fixed positive integer.
Note that
\[ \#\left\{p_{m+1}^2\prod_{i=1}^{m}p_i\leq X : p_i \text{ are distinct primes in } A  \right\}=-(m+1)\pi_{m+1,A}(X)+\sum_{\substack{p\leq X \\ p\in A}} \pi_{m,A}\left(\frac{X}{p}\right). \]
From this, we obtain that
\begin{equation}\label{e : lemma 2.2-1}
\begin{aligned}
-(m+1)\pi_{m+1,P}(X)+\sum_{\text{prime }p\leq X} \pi_{m,P}\left(\frac{X}{p}\right)&\geq -(m+1)\pi_{m+1,A}(X)+ \sum_{\substack{p\leq X \\ p\in A}} \pi_{m,A}\left(\frac{X}{p}\right)\\
&\geq 0.
\end{aligned}
\end{equation}
From the result of Landau \cite[pp. 208-209]{LA}, we obtain that
\begin{equation}\label{eq 1}
    -(m+1)\pi_{m+1,P}(X)+\sum_{\text{prime }p\leq X}\pi_{m,P}\left(\frac{X}{p}\right)=o\left(\frac{X(\log \log X)^{m-1}}{\log X} \right).
\end{equation}
Thus, we have by (\ref{e : lemma 2.2-1})
\begin{equation}\label{e: lemma 2.2-4}
     \pi_{m+1,A}(X)=\frac{1}{m+1}\sum_{\substack{p\leq X\\ p\in A}}\pi_{m,A}\left(\frac{X}{p}\right)+o\left(\frac{X(\log \log X)^{m-1}}{\log X}\right).
\end{equation}

Let us consider the asymptotic behavior of
\begin{equation*}
    \sum_{\substack{p\leq X \\ p\in A}}\pi_{k_0,A}\left(\frac{X}{p}\right)
\end{equation*}
as $X\to \infty$.
By the induction hypothesis, we have
\[ \lim_{X\to \infty} \frac{\pi_{k_0,A}(X)}{\pi_{k_0,P}(X)}=\mu^{k_0}. \]
Thus, given a positive real number $\epsilon$, there is a constant $M(k_0,\epsilon)>0$ such that for all $X>M(k_0,\epsilon)$,
\[ \left|\frac{\pi_{k_0,A}(X)}{\pi_{k_0,P}(X)}-\mu^{k_0} \right|<\frac{\mu^{k_0}\epsilon}{3}. \]
Then, we get
\begin{equation*}
    \begin{aligned}
        \sum_{p\leq X}\pi_{k_0,A}\left(\frac{X}{p}\right)&=\sum_{p\leq \frac{X}{M(k_0,\epsilon)}}\pi_{k_0,A}\left(\frac{X}{p}\right)+\sum_{p>\frac{X}{M(k_0,\epsilon)}}\pi_{k_0,A}\left(\frac{X}{p}\right)\\
        &\leq \sum_{p\leq \frac{X}{M(k_0,\epsilon)}}\pi_{k_0,A}\left(\frac{X}{p}\right)+\pi_{1,P}(X)\pi_{k_0,A}(M(k_0,\epsilon))\\
        &\leq \mu^{k_0}\left(1+\frac{\epsilon}{3}\right)\sum_{p\leq \frac{X}{M(k_0,\epsilon)}}\pi_{k_0,P}\left(\frac{X}{p}\right)+\pi_{1,P}(X)\pi_{k_0,A}(M(k_0,\epsilon))\\
        &\leq \mu^{k_0}\left(1+\frac{\epsilon}{3}\right)\sum_{p\leq X}\pi_{k_0,P}\left(\frac{X}{p}\right)+\pi_{1,P}(X)\pi_{k_0,A}(M(k_0,\epsilon)).
    \end{aligned}
\end{equation*}
By \eqref{eq 1}, we have
\[ \sum_{p\leq X}\pi_{k_0,P}\left(\frac{X}{p}\right)=(k_0+1)\pi_{k_0+1,P}(X)+o\left(\frac{X(\log \log X)^{k_0-1}}{\log X}\right). \]
For each positive integer $\nu$, Landau \cite[pp. 211]{LA} proved that 
\[ \pi_{\nu,P}(X)\sim \frac{X(\log \log X)^{\nu-1}}{(\nu-1)!\log X}. \]
Here, we write $f(X)\sim g(X)$ if
\[ \lim_{X\to\infty} \frac{f(X)}{g(X)}=1. \]
Thus, we obtain that
\[ \sum_{p\leq X}\pi_{k_0,P}\left(\frac{X}{p}\right)\sim (k_0+1)\frac{X(\log \log X)^{k_0}}{k_0! \log X}. \]
Since $\pi_{1,P}(X)\sim \frac{X}{\log X}$, it follows that there is a constant $N(k_0,\epsilon)$ such that for all $X>N(k_0,\epsilon)$,
\[ \pi_{1,P}(X)\pi_{k_0,A}(M(k_0,\epsilon))<\frac{\mu^{k_0}\epsilon}{3}\sum_{p\leq X}\pi_{k_0,P}\left(\frac{X}{p}\right). \]
Hence, we obtain that
\[ \sum_{p\leq X}\pi_{k_0,A}\left(\frac{X}{p}\right)\leq \mu^{k_0}\left(1+\frac{2\epsilon}{3}\right)\sum_{p\leq X}\pi_{k_0,P}\left(\frac{X}{p}\right). \]

If $X>N(k_0,\epsilon)$, then we have
\begin{equation*}
    \begin{aligned}
        \sum_{p\leq X}\pi_{k_0,A}\left(\frac{X}{p}\right)&\geq \sum_{p\leq \frac{X}{M(k_0,\epsilon)}}\pi_{k_0,A}\left(\frac{X}{p}\right)\\
        &\geq \mu^{k_0}\left(1-\frac{\epsilon}{3}\right)\sum_{p\leq \frac{X}{M(k_0,\epsilon)}}\pi_{k_0,P}\left(\frac{X}{p}\right)\\
        &=\mu^{k_0}\left(1-\frac{\epsilon}{3}\right)\left(\sum_{p\leq X}\pi_{k_0,P}\left(\frac{X}{p}\right) - \sum_{p>\frac{X}{M(k_0,\epsilon)}}\pi_{k_0,P}\left(\frac{X}{p}\right)\right)\\
        &\geq \mu^{k_0}\left(1-\frac{\epsilon}{3}\right)\sum_{p\leq X}\pi_{k_0,P}\left(\frac{X}{p}\right) - \sum_{p>\frac{X}{M(k_0,\epsilon)}}\pi_{k_0,P}\left(\frac{X}{p}\right)\\
        &\geq\mu^{k_0}\left(1-\frac{2\epsilon}{3}\right)\sum_{p\leq X}\pi_{k_0,P}\left(\frac{X}{p}\right).
    \end{aligned}
\end{equation*}
Thus, we obtain that
\begin{equation*}
    \sum_{p\leq X}\pi_{k_0,A}\left(\frac{X}{p}\right) \sim \mu^{k_0}\sum_{p\leq X}\pi_{k_0,P}\left(\frac{X}{p}\right)\sim \mu^{k_0}(k_0+1)\pi_{k_0+1,P}(X)
\end{equation*}
and that
\begin{equation}\label{eq 2}
\begin{aligned}
    \sum_{p\leq X} \pi_{k_0,A}\left(\frac{X}{p}\right) = \mu^{k_0}(k_0+1)\pi_{k_0+1,P}(X)+o\left(\frac{X(\log \log X)^{k_0}}{\log X} \right).
\end{aligned}
\end{equation}

We claim that 
\begin{equation}\label{eq 8}
    \sum_{\substack{p\leq X \\ p\in A}} \pi_{k_0,A}\left(\frac{X}{p}\right) =\mu\sum_{p\leq X} \pi_{k_0,A}\left(\frac{X}{p}\right)+ o\left(\frac{X(\log \log X)^{k_0}}{\log X}\right).
\end{equation}
Then, we have the following by (\ref{e: lemma 2.2-4}), \eqref{eq 2}, and \eqref{eq 8}:
\begin{equation*}
    \begin{aligned}
    \pi_{k_0+1,A}(X)&=\frac{1}{k_0+1}\sum_{\substack{p\leq X \\ p\in A}}\pi_{k_0,A}\left(\frac{X}{p}\right)+o\left(\frac{X(\log \log X)^{k_0-1}}{\log X}\right)\\
    &=\frac{\mu}{k_0+1}\sum_{p\leq X}\pi_{k_0,A}\left(\frac{X}{p}\right)+o\left(\frac{X(\log \log X)^{k_0}}{\log X}\right)\\
    &=\mu^{k_0+1}\pi_{k_0+1,P}(X)+o\left(\frac{X(\log \log X)^{k_0}}{\log X}\right)\\
    &\sim \mu^{k_0+1}\pi_{k_0+1,P}(X).
    \end{aligned}
\end{equation*}

To complete the proof, we prove the claim.
By Abel's summation formula, we obtain
\[ \sum_{\substack{p\leq X \\ p\in A}}\pi_{k_0,A}\left(\frac{X}{p}\right)=\sum_{n=1}^{X-1} \pi_{1,A}(n)\left(\pi_{k_0,A}\left(\frac{X}{n}\right)-\pi_{k_0,A}\left(\frac{X}{n+1}\right) \right)+\pi_{1,A}(X)\pi_{k_0,A}(1) \]
and
\[ \sum_{p\leq X}\pi_{k_0,A}\left(\frac{X}{p}\right)=\sum_{n=1}^{X-1} \pi_{1,P}(n)\left(\pi_{k_0,A}\left(\frac{X}{n}\right)-\pi_{k_0,A}\left(\frac{X}{n+1}\right) \right)+\pi_{1,P}(X)\pi_{k_0,A}(1). \]
Note that
\begin{equation*}
    \begin{aligned}
    \left|\sum_{n=1}^{[\sqrt{\log \log X}]} \pi_{1,P}(n)\left(\pi_{k_0,A}\left(\frac{X}{n}\right)-\pi_{k_0,A}\left(\frac{X}{n+1}\right) \right) \right|&\leq \sum_{n=1}^{[\sqrt{\log \log X}]}\pi_{1,P}(n)\pi_{k_0,A}(X)\\
    &\leq \pi_{1,P}\left(\sqrt{\log \log X} \right)\sqrt{\log \log X}\pi_{k_0,A}(X)\\
    &=o\left(\frac{X(\log \log X)^{k_0}}{\log X} \right).
    \end{aligned}
\end{equation*}
Thus, we have
\begin{equation*}
    \sum_{p\leq X}\pi_{k_0,A}\left(\frac{X}{p}\right)=\sum_{n>[\sqrt{\log \log X}]}^{X-1}\pi_{1,P}(n)\left(\pi_{k_0,A}\left(\frac{X}{n}\right)-\pi_{k_0,A}\left(\frac{X}{n+1}\right) \right) + o\left(\frac{X(\log \log X)^{k_0}}{\log X}\right).
\end{equation*}
Since $\mu\cdot \pi_{1,P}(n)\sim \pi_{1,A}(n)$ and $\sum_{p\leq X}\pi_{k_0,A}\left(
\frac{X}{p}\right)\sim \frac{\mu^{k_0}(k_0+1)}{k_0!}\frac{X(\log \log X)^{k_0}}{\log X},$ we have
\[ \sum_{\substack{p\leq X \\ p\in A}} \pi_{k_0,A}\left(\frac{X}{p}\right)= \mu\sum_{p\leq X} \pi_{k_0,A}\left(\frac{X}{p}\right)+o\left(\frac{X(\log \log X)^{k_0}}{\log X} \right). \]

\end{proof}

Now, we are ready to prove Theorem \ref{thm : thm of weak}.

\begin{proof}[Proof of Theorem \ref{thm : thm of weak}]
Let $M:=\prod_{\ell\mid M}\ell^{v_{\ell}}$ be a positive integer such that $(M,N)=1$, where $\ell$'s denote primes and $v_{\ell}$'s denote positive integers.
We choose a set of representatives $\{r_1,\dots, r_{s(f)}\}$ of $\overline{\Omega(f)}$ such that every $r_i$ is a square-free integer.
Let $\mathcal{P}(f)$ be the set of odd primes $\ell$ such that $(\ell,Nr_i)=1$ and $(\frac{r_i}{\ell})$ are the same for all $i\in \{1,\dots,s(f)\}$.
We have by Theorem \ref{thm : thm of Tr} that if $\ell\in \mathcal{P}(f)$ and $\ell\mid M$, then there is a positive integer $\beta\geq v_{\ell}-1$ and a cusp form
\[f_{\ell}\in S_{\frac{k+\ell^{\beta}(\ell^2-1)}{2}}\left(\Gamma_0(N\ell^2),\chi\right)\cap \mathbb{Z}[[q]] \]
such that
\begin{equation*}
    f_{\ell}(z)\equiv  \sum_{\left(\frac{r_1n}{\ell}\right)=-1}a_{f}(n)q^n  \pmod{\ell^{v_{\ell}}}.
\end{equation*}
For a positive integer $n$, let $\mathcal{P}_{n}$ be a subset of $\mathcal{P}(f)$ consisting of primes $\ell$ such that
\[ \left(\frac{r_1n}{\ell}\right)=-1 \text{ and } \ell \nmid a_{f}(n). \]
Let $n_0$ be a square-free positive integer with $n_0\nmid N$, and let $m$ be a non-zero integer.
Assume that $\ell\in \mathcal{P}_{n_0m^2}$ for every prime divisor $\ell$ of $M$.
Then, we have $n_0\nmid NM$ and $(n_0m^2a_{f}(n_0m^2),M)=1$.
If $k$ is odd, then Lemmas \ref{lem : lem of distribution-half} and \ref{lem : main lem 1} imply $M\in \mathcal{N}(f)$.
If $k$ is even, then by Lemma \ref{lem : lem of distribution_integral}, we have $M\in \mathcal{N}(f)$.
Thus, for each pair $(n_0,m)$, let us consider the number of $M \leq X$ such that $\ell\in \mathcal{P}_{n_0m^2}$ for every prime divisor $\ell$ of $M$.

Let $B_{d}^{sq}$ be the subset of $B_{d}$ consisting of square-free integers.
Note that
\[\#\{M\leq X : M\in B_{d} \} \sim \#\{M\leq X : M\in B_{d}^{sq} \}\]
(see \cite[Chapter II.6]{TE}).
For each positive integer $n$, let $B_{d,n}^{sq}$ be the subset of $B_{d}^{sq}$ consisting of $M$ such that every prime divisor $\ell$ of $M$ satisfies $\ell\in \mathcal{P}(f)$ and $\left(\frac{r_1n}{\ell}\right)=-1$.
For each odd prime $\ell$, note that $\ell\in \mathcal{P}(f)$ and $\left(\frac{r_1n}{\ell}\right)=-1$ if and only if $(\ell,N)=1$ and $\left(\frac{r_i}{\ell}\right)=-\left(\frac{n}{\ell}\right)$ for all $i\in \{1,\dots,s(f)\}$. 
This implies that the density of primes $\ell$ such that $\ell\in \mathcal{P}(f)$ and $\left(\frac{r_1n}{\ell}\right)=-1$ is $\left(\frac{1}{2}\right)^{s(f)}$.
Then, by Lemma \ref{lem : density of k-primes}, we have
\[ \lim_{X\to \infty} \frac{\#\{M\leq X : M\in B_{d,n}^{sq} \}}{\pi_{d,P}(X)}=\left(\frac{1}{2}\right)^{d s(f)}.  \]

By assumption $f$ is not a linear combination of single-variable theta series.
By \cite[Th\'{e}or\`{e}me 3]{V}, there are infinitely many square-free integers $n_i\nmid N$, such that $a_{f}(n_im_i^2)\neq 0$ for some positive integer $m_i$.
By Lemma \ref{lem : density of k-primes} and the principle of inclusion and exclusion, we have the following: for each positive integer $t$,
\[ \lim_{X\to \infty} \frac{\#\left\{M\leq X :M\in \cup_{i=1}^{t} B_{d,n_im_i^2}^{sq}  \right\}}{\pi_{d,P}(X)}=\left(1-\left(1-\frac{1}{2^d}\right)^t\right)\left(\frac{1}{2}\right)^{d(s(f)-1)}. \]
For a positive real number $\epsilon$ with $\epsilon<1$, we take a positive integer $t$ such that
\[ \left(1-\frac{1}{2^d}\right)^t<\epsilon. \]
Note that if $M\in B^{sq}_{d,n}$ and $(M,a_{f}(n))=1$, then the set of prime divisors of $M$ is contained in $\mathcal{P}_{n}$.
Thus, if there is a square-free integer $n_i\nmid N$, such that $M\in B_{d,n_im_i^2}^{sq}$ and $(M,a_{f}(n_im_i^2))=1$ for some positive integer $m_i$, then $M\in \mathcal{N}(f)$.
From this, we have
\begin{equation*}
    \begin{aligned}
    \#\{M\leq X : M\in B_{d}^{sq}\cap \mathcal{N}(f) \}&\geq \#\left\{M\leq X : M\in \cup_{i=1}^{t} B_{d,n_im_i^2}^{sq} \right\}\\
    &-\sum_{i=1}^{t}\#\left\{M\leq X : M\in B_{d,n_im_i^2}^{sq} \text{ and }\left(M,a_{f}(n_im_i^2)\right)\neq 1 \right\}.
    \end{aligned}
\end{equation*}
By Landau's result \cite[Chapter 56]{LA}, we deduce that for each positive integer $n$,
\begin{equation*}
    \begin{aligned}
    \#\left\{M\leq X : M\in B_{d,n}^{sq} \text{ and }\left(M,a_{f}(n)\right)\neq 1 \right\}&\leq \sum_{j=1}^{u} \#\left\{M\leq \frac{X}{p_{j}} :M\in B_{d-1,n}^{sq} \right\}\\
    &=O\left(\frac{X(\log \log X)^{d-2}}{\log X} \right),
    \end{aligned}
\end{equation*}
where $\{p_1,\dots, p_{u}\}$ is the set of prime divisors of $a_{f}(n)$.
Then, we have
\begin{equation*}
    \begin{aligned}
    \liminf_{X\to \infty}\frac{\#\{ M\leq X : M\in B_{d}^{sq}\cap \mathcal{N}(f) \}}{\pi_{d,P}(X)}&\geq  \liminf_{X\to \infty} \frac{\#\left\{ M\leq X : M\in \cup_{i=1}^{t} B_{d,n_im_i^2}^{sq} \right\}}{\pi_{d,P}(X)}\\
    &=\left(1-\left(1-\frac{1}{2^d}\right)^t\right)\left(\frac{1}{2}\right)^{d(s(f)-1)}\\
    &\geq(1-\epsilon)\left(\frac{1}{2}\right)^{d(s(f)-1)}.
    \end{aligned}
\end{equation*}
Therefore, the proof is complete.
\end{proof}

Let us recall that the generating function of $p(n)$ is given by
\[ \sum_{n=0}^{\infty} p(n)q^n=\prod_{n=1}^{\infty} \frac{1}{1-q^n}. \]
For convenience, let $p(n):=0$ if $n$ is not a non-negative integer.
From this, we have
\[ \frac{1}{\eta(24z)}=\sum_{n=-1}^{\infty} p\left(\frac{n+1}{24}\right)q^n. \]
Note that $\frac{1}{\eta(24z)}$ is a weakly holomorphic modular form of weight $-\frac{1}{2}$ on $\Gamma_0(576)$ with nebentypus. This fact will permit us to use Theorem \ref{thm : thm of weak} to get the proof of Theorem \ref{thm : thm 2}.
\begin{proof}[Proof of Theorem \ref{thm : thm 2}]
Let $f(z):=\frac{1}{\eta(24z)}$ and note that $\Delta(z):=\eta^{24}(z)$ is a cusp form of weight $12$ on $\SL_2(\mathbb{Z})$.
For each $\rho:=\left(\begin{smallmatrix}
a & b\\
c & d
\end{smallmatrix}\right)\in \SL_2(\mathbb{Z})$,
we deduce by \cite[Lemma 3.2.5]{LI} that
\[ \Delta(24z)|_{12}\rho=q^{\frac{(c,24)^2}{24}}\left(a_0+a_1 q^{\frac{(c,24)^2}{24}}+\cdots \right), \]
where $a_0\neq 0$ and $a_1\in \CC$.
From this, we obtain
\begin{equation}\label{e: Fourier exp of eta0}
    f|_{-\frac{1}{2}} \rho=q^{-\frac{(c,24)^2}{576}}\left(a_{f}\left(\rho : -\frac{(c,24)^2}{576} \right)+a_{f}\left(\rho : \frac{23(c,24)^2}{576}\right)q^{\frac{(c,24)^2}{24}}+\cdots\right).
\end{equation}
It follows that
\begin{equation}\label{eq 4}
    \Omega(f)\subset \left\{-\frac{(c,24)^2}{576} : c\in \mathbb{Z}\right\}.
\end{equation}
Thus, $\{-1\}$ is a set of representatives of $\overline{\Omega(f)}$.
Therefore, Theorem \ref{thm : thm 2} follows from Theorem \ref{thm : thm of weak} since $s(f)=1$, which makes the right side of the inequality in Theorem \ref{thm : thm of weak} equal to one.
\end{proof}
As another application of Theorem \ref{thm : thm of weak}, we study the distribution of the number of generalized Frobenius partitions with $h$-colors modulo an integer $M$.
For convenience, let $c\phi_h(n):=0$ if $n$ is not a non-negative integer.
Recall that the generating function of $c\phi_h(n)$ is given by
\[  C\Phi_h(q):=\sum_{n=0}^{\infty} c\phi_h(n)q^n=\frac{1}{\prod_{n=1}^{\infty}(1-q^n)^h}\sum_{m\in \mathbb{Z}^{h-1}}q^{Q(m)}, \]
where $Q(m_1,\dots, m_{h-1}):=\sum_{i=1}^{h-1} m_i^2+\sum_{1\leq i<j\leq h-1} m_im_j$.
Let $A_{h}:=12(I_{h-1}+J_{h-1})$ and $g_{h}(z):=\sum_{m\in \mathbb{Z}^{h-1}} q^{m^{\top}A_{h}m}$, where $I_{h-1}$ is the identity matrix of size $h-1$ and $J_{h-1}$ is the matrix of size $(h-1)\times(h-1)$ with all entries being $1$. Then,
\[ q^{-h}C\Phi_h(q^{24})=\sum_{n=-h}^{\infty} c\phi_{h}\left(\frac{n+h}{24}\right)q^n=\frac{g_h(z)}{\eta^{h}(24z)}.\]

For $v\in \ZZ^{h-1}$, let
\[ \theta(z;v,2A_h,24h):=\sum_{\substack{w\in \ZZ^{h-1} \\ w\equiv v\pmod{24h}}}e\left(\frac{w^{\top}A_hw}{576h^2}z\right).\]
Then, we have
\[ g_h(z)=\theta(z;0,2A_h,24h). \]
Let $\rho:=\left(\begin{smallmatrix}
a & b\\
c & d
\end{smallmatrix}\right)$ be in $\SL_2(\ZZ)$.
Note that for $w,w_0$ and $x$ in $\ZZ^{h-1}$, we have
\begin{multline*}
    e\left(\frac{576ah^2 x^{\top}A_h x+48h(w+24hw_0)^{\top}A_h x+d(w+24hw_0)^{\top}A_h(w+24hw_0)}{576ch^2}\right)\\
    =e\left(\frac{576ah^2 (x+dw_0)^{\top}A_h (x+dw_0)+48hw^{\top}A_h (x+dw_0)+dw^{\top}A_hw}{576ch^2}\right).
\end{multline*}
Thus, for $w\in \ZZ^{h-1}$ with $A_{h}w\equiv 0\pmod{12h}$, let $\overline{w}$ be the image of $w$ in $(\ZZ/24h\ZZ)^{h-1}$ and
\begin{equation}\label{e : eqfix4}
    \begin{aligned}
    \Phi_{\rho}(0,\overline{w}):=\sum_{x\in (\ZZ/c\ZZ)^{h-1}}e\left(\frac{576ah^2 x^{\top}A_h x+48hw^{\top}A_h x+dw^{\top}A_hw}{576ch^2} \right).
    \end{aligned}
\end{equation}
By the transformation formula for theta series in \cite[Chapter 4.9]{MI}, we have
\begin{equation*}
    \begin{aligned}
    g_h|_{\frac{h-1}{2}}\rho&=A(\rho)\sum_{\substack{\overline{v}\in (\ZZ/24h\ZZ)^{h-1}\\ A_h v\equiv 0\pmod{12h}}}\Phi_{\rho}(0,\overline{v})\theta(z;v,2A_h,24h)\\
    &=A(\rho)\sum_{\substack{\overline{v}\in (\ZZ/24h\ZZ)^{h-1}\\ A_h v\equiv 0\pmod{12h}}}\Phi_{\rho}(0,\overline{v})\left(\sum_{\substack{w\in \ZZ^{h-1} \\ \overline{w}= \overline{v}}}e\left(\frac{w^{\top}A_hw}{576h^2}z\right) \right)\\
    &=A(\rho) \sum_{\substack{w\in \ZZ^{h-1}\\ A_hw\equiv 0\pmod{12h}}} \Phi_{\rho}(0,\overline{w})e\left(\frac{w^{\top}A_hw}{576h^2}z\right),
    \end{aligned}
\end{equation*}
where $A(\rho)\neq 0$.
Note that for $w\in \ZZ^{h-1}$ with $A_hw\equiv 0\pmod{12h}$, we have
\[ \frac{w^{\top}A_hw}{576h^2}\in \frac{1}{48h}\mathbb{Z}. \]
Thus, we have
\begin{equation}\label{e : eqfix2}
   g_{h}|_{\frac{h-1}{2}}\rho=\sum_{n\gg -\infty} a_{g_h}\left(\rho: \frac{n}{48h}\right)q^{\frac{n}{48h}},
\end{equation}
where
\[ a_{g_h}\left(\rho : \frac{n}{48h}\right)=A(\rho)\sum_{\substack{w\in \ZZ^{h-1}\\ A_hw\equiv 0\pmod{12h}\\ w^{\top}A_h w =12nh}} \Phi_{\rho}(0,\overline{w}).\]

Let us note that
\[ 12h(A_{h})^{-1}\in M_{h-1}(\mathbb{Z})\] since $\det(I_{h-1}+J_{h-1})=h$.
Thus, $g_h(z)=\theta(z;0,2A_h,24h)$ is a weakly holomorphic modular form of weight $\frac{h-1}{2}$ on $\Gamma_0(48h)$ with nebentypus (for details, see \cite[Theorem 4.9.3]{MI}).
Since
\[q^{-h}C\Phi_{h}(q^{24})=\frac{g_h(z)}{\eta^{h}(24z)},\]
it follows that $q^{-h}C\Phi_{h}(q^{24})$ is a weakly holomorphic modular form of weight $-\frac{1}{2}$ on $\Gamma_0(576h)$ with nebentypus.

With the modular property of $q^{-h}C\Phi_{h}(q^{24})$ in hand, we prove Theorem \ref{thm : thm 3}.
\begin{proof}[Proof of Theorem \ref{thm : thm 3}]
To prove this theorem, we apply Theorem \ref{thm : thm of weak} to $q^{-h}C\Phi_{h}(q^{24})=\frac{g_h(z)}{\eta^{h}(24z)}$. Thus, to get $\overline{\Omega(q^{-h}C\Phi_{h}(q^{24}))}$,
let us consider the Fourier expansions of $\frac{1}{\eta^{h}(24z)}|_{-\frac{h}{2}}\rho$ and $g_{h}|_{\frac{h-1}{2}}\rho$ for $\rho:=\left(\begin{smallmatrix}
a & b\\
c &d
\end{smallmatrix}\right)\in \SL_2(\ZZ)$.

By \eqref{e: Fourier exp of eta0}, we have
\begin{equation}\label{e: Fourier exp of eta}
    \frac{1}{\eta^{h}(24z)}|_{-\frac{h}{2}} \rho=q^{-\frac{(c,24)^2}{576}h}\left(b_0+b_1q^{\frac{(c,24)^2}{24}}+\cdots\right),
\end{equation}
with $b_0\neq 0$ and $b_1\in \CC$.
Combining \eqref{e : eqfix2} and \eqref{e: Fourier exp of eta}, we obtain
\begin{equation}\label{e : eq fix1}
    q^{-h}C\Phi_{h}(q^{24})|_{-\frac{1}{2}}\rho=\alpha q^{-\frac{(c,24)^2}{576}h}\left(b_0+b_1q^{\frac{(c,24)^2}{24}}+\cdots\right)\left(\sum_{n=0}^{\infty} a_{g_h}\left(\rho : \frac{n}{48h}\right)q^{\frac{n}{48h}}\right)
\end{equation}
with $\alpha\neq 0$.
Thus, $q^{-h}C\Phi_{h}(q^{24})|_{-\frac{1}{2}}\rho$ has the Fourier expansion of the form
\[
 q^{-h}C\Phi_{h}(q^{24})|_{-\frac{1}{2}}\rho=\alpha b_0a_{g_h}(\rho : 0)q^{-\frac{(c,24)^2}{576}h}+\cdots,
\]
and for every $\rho$, the period of $q^{-h}C\Phi_{h}(q^{24})|_{-\frac{1}{2}}\rho$ divides $576h$. Thus,
\[ \Omega(q^{-h}C\Phi_{h}(q^{24})) \subset \left\{-\frac{j}{576h} : \text{ positive integers } j\leq 576h^2 \right\}. \]
Hence, we conclude that $\# \overline{\Omega(q^{-h}C\Phi_{h}(q^{24}))}\leq 576h^2$.

In particular, when $h\in \{2,3\}$, we prove $\#\overline{\Omega(q^{-h}C\Phi_{h}(q^{24}))}=1$.
Assume that $h=2$. Since $h<24$, we have $\frac{(c,24)^2}{24}-\frac{(c,24)^2h}{576}>0$.
By \eqref{e : eq fix1}, the function $q^{-h}C\Phi_{h}(q^{24})|_{-\frac{1}{2}}\rho$ has the Fourier expansion of the form
\[ q^{-h}C\Phi_{h}(q^{24})|_{-\frac{1}{2}}\rho=\alpha b_0\sum_{n=0}^{\infty} a_{g_h}\left(\rho : \frac{n}{48h}\right)q^{\frac{n}{48h}-\frac{(c,24)^2h}{576}} + \sum_{m=1}^{\infty} c\left(\rho:\frac{m}{576h}\right)q^{\frac{m}{576h}}, \]
where $c\left(\rho:\frac{m}{576h}\right)\in \CC$.
It follows that
\begin{equation}\label{e : eqfix5}
    \Omega\left(q^{-h}C\Phi_{h}(q^{24})\right)=\left\{\frac{n}{48h}-\frac{(c,24)^2h}{576}<0 : n\in \mathbb{Z} \, \text{ and } a_{g_h}\left(\rho : \frac{n}{48h}\right)\neq 0 \text{ for some }\rho\in \SL_2(\ZZ) \right\}.
\end{equation}

Considering the Fourier expansion of $g_{h}|_{\frac{h-1}{2}}\rho$ in \eqref{e : eqfix2}, we have that if $a_{h}(\rho:\frac{n}{48h})\neq 0$, then there exists $w\in \ZZ^{h-1}$ such that
\begin{equation}\label{e : eqfix3}
    A_{h}w\equiv 0\pmod{12h}, \ w^{\top}A_{h}w=12nh, \text{ and } \Phi_{\rho}(0,\overline{w})\neq 0.
\end{equation}
By the definition of $\Phi_{\rho}(0,\overline{w})$ in \eqref{e : eqfix4}, we have
\begin{equation*}
    \begin{aligned}
    \Phi_{\rho}(0,\overline{w})&=e\left(\frac{dw^2}{96c} \right)\sum_{x=0}^{c-1}e\left(\frac{24ax^2+wx}{c}\right)\\
    &= e\left(\frac{dw^2}{96c}\right)\sum_{v=0}^{\frac{c}{(c,24)}-1}\sum_{u=0}^{(c,24)-1}e\left(\frac{1}{c}\left(24a\left(\frac{c}{(c,24)}u+v\right)^2+w\left(\frac{c}{(c,24)}u+v \right) \right) \right)\\
    &=e\left(\frac{dw^2}{96c}\right)\sum_{v=0}^{\frac{c}{(c,24)}-1}\sum_{u=0}^{(c,24)-1}e\left(\frac{24acu^2}{(c,24)^2}+\frac{48auv}{(c,24)}+\frac{24av^2}{c}+\frac{wu}{(c,24)}+\frac{wv}{c} \right).
    \end{aligned}
\end{equation*}
Since $(c,24)^2\mid 24c$ and $(c,24)\mid 48$, we have
\begin{equation*}
\begin{aligned}
\Phi_{\rho}(0,\overline{w})&=e\left(\frac{dw^2}{96c}\right)\sum_{v=0}^{\frac{c}{(c,24)}-1}\sum_{u=0}^{(c,24)-1} e\left(\frac{24av^2}{c}+\frac{wu}{(c,24)}+\frac{wv}{c}\right)\\
&=e\left(\frac{dw^2}{96c}\right)\sum_{v=0}^{\frac{c}{(c,24)}-1}e\left(\frac{24av^2}{c}+\frac{wv}{c}\right) \sum_{u=0}^{(c,24)-1}e\left(\frac{wu}{(c,24)}\right).
\end{aligned}
\end{equation*}
Thus, if $\Phi_{\rho}(0,\overline{w})\neq 0$, then $(c,24)\mid w$.
Since $h=2$ and $A_h=24$, the condition \eqref{e : eqfix3} implies that
\[
w^2=n \text{ and } (c,24)\mid w.
\]
Thus, we have by \eqref{e : eqfix5} and \eqref{e : eqfix3} that every element in $\Omega\left(q^{-2}C\Phi_{2}(q^{24})\right)$ is equal to $\frac{n}{96}-\frac{(c,24)^2}{288}$ for some positive integers $c$ and $n$ such that there exists an integer $w$ such that $n=w^2$ and $(c,24)\mid w$.
Note that if $w$ is a non-zero integer with $(c,24)\mid w$, then $\frac{w^2}{96}-\frac{(c,24)^2}{288}$ is positive.
Thus, we conclude that
\[ \Omega\left(q^{-2}C\Phi_2(q^{24})\right)\subset \left\{-\frac{(c,24)^2}{288} : c\in \mathbb{Z}\right\}. \]
Therefore, $\{-\frac{1}{2}\}$ is a set of representatives of $\overline{\Omega(q^{-2}C\Phi_2(q^{24}))}$.
Hence Theorem \ref{thm : thm of weak} implies the desired result pertaining to $h = 2$ in Theorem \ref{thm : thm 3}.

Similar to the case $h=2$, we have
\[ \Omega(q^{-3}C\Phi_{3}(q^{24}))\subset\left\{-\frac{(c,24)^2}{192} : c\in \mathbb{Z}\right\}\cup \left\{-\frac{(c,24)^2}{1728} : c\in \mathbb{Z}\right\}. \]
We therefore omit the proof of the case $h=3$.
\end{proof}

\subsection*{Acknowledgments.}  The authors are grateful to the referees for their careful reading and helpful comments, which have improved the previous version of this paper. 
The first author was supported by the National Research Foundation of Korea (NRF) grant funded by the Korea government (MSIT)(RS-2024-00334203).
The second author was supported by a KIAS Individual Grant (MG086302) at Korea Institute for Advanced Study.

\end{document}